\documentclass[11pt]{amsart}

\usepackage{epigamath}

\usepackage[french]{babel}
\numberwithin{equation}{section}

\usepackage[T1]{fontenc}

\usepackage{amsmath}
\usepackage{amsthm}
\usepackage{amssymb}
\usepackage{tikz}
\usepackage{tikz-cd}

\usepackage{imakeidx} 
\usepackage{tikz-cd}
\usepackage{verbatim}
\usepackage{enumitem}
\usepackage{multicol}

\usepackage{adjustbox}

\usepackage{mathtools}

\newtheorem{thm}{Th\'eor\`eme}[section]

\newtheorem{lemma}[thm]{Lemme}
\newtheorem{cor}[thm]{Corollaire}

\newtheorem{question}[thm]{Question}

\theoremstyle{definition}

\theoremstyle{remark}

\newtheorem{rmk}[thm]{Remarque}
\newcommand{\Q}{\mathbb Q}
\newcommand{\F}{\mathbb F}

\newcommand{\Z}{\mathbb Z}
\newcommand{\G}{\mathbb G}

\renewcommand{\P}{\mathbb P}

\newcommand{\mc}[1]{\mathcal{#1}}
\newcommand{\cl}{\overline}
\newcommand{\set}[1]{\left\{#1\right\}}
\renewcommand{\phi}{\varphi}

\newcommand{\on}[1]{\operatorname{#1}}

\EpigaVolumeYear{6}{2022} \EpigaArticleNr{11} 
\ReceivedOn{October 5, 2021}
\InFinalFormOn{February 6, 2022}
\AcceptedOn{February 17, 2022}

\title{Autour de la conjecture de Tate enti\`ere pour certains produits de dimension 3 sur un corps fini}
\titlemark{La conjecture de Tate enti\`ere pour certains produits}

\author{Federico Scavia}
\address{Department of Mathematics,
	University of California,
	Los Angeles, CA 90095, United States of America}
\email{scavia@math.ucla.edu}

\authormark{F. Scavia}

\TitleInEnglish{On the integral Tate conjecture for certain products of dimension 3 over a finite field}

\AbstractInEnglish{Let $X$ be the product of a surface satisfying $b_2=\rho$ and of a curve over a finite field. We study a strong form of	the integral Tate conjecture for $1$-cycles on $X$. We generalize and give unconditional proofs of several results of our previous paper with J.-L. Colliot-Th\'{e}l\`ene.}

\AbstractInFrench{Soit $X$ le produit d'une  surface satisfaisant $b_2=\rho$ et d'une courbe sur un corps fini. On \'etudie une forme forte de la conjecture  de Tate enti\`ere pour les $1$-cycles sur $X$. On g\'en\'eralise et donne des preuves inconditionnelles de plusieurs r\'esultats de notre article pr\'ec\'edent avec J.-L.~Colliot-Th\'{e}l\`ene.}



\begin{document}

%%%%%%%%%%%%%%%%%%%%%%%%%%%%%%%
% Title page
%%%%%%%%%%%%%%%%%%%%%%%%%%%%%%%

%\removeabove{}
%\removebetween{}
%\removebelow{}

\maketitle

\begin{prelims}

\DisplayAbstractInFrench

\bigskip

\DisplayKeyWordsFr

\medskip

\DisplayMSCclassFr

\bigskip

\languagesection{English}

\bigskip

\DisplayTitleInEnglish

\medskip

\DisplayAbstractInEnglish

\end{prelims}

%%%%%%%%%%%%%%%%%%%%%
% Table of Contents
%%%%%%%%%%%%%%%%%%%%%

\newpage

\setcounter{tocdepth}{1}

\tableofcontents

%%%%%%%%%%%%%%%%%%%%%
% Content begins here
%%%%%%%%%%%%%%%%%%%%%

\section{Introduction}

	%\subsection{Conjectures de Tate} 
	Soient $\F$ un corps fini, $\cl{\F}$ une cl\^oture alg\'ebrique de $\F$, $G$ le groupe de Galois absolu $\on{Gal}(\cl{\F}/\F)$, $\ell$ un nombre premier inversible dans $\F$, $X$ une $\F$-vari\'et\'e projective, lisse et g\'eom\'etriquement connexe, de dimension $d$, et $\cl{X}:=X\times_{\F}\cl{\F}$. Si $M$ est un $G$-module, on note $M^{(1)}\subset M$ le sous-groupe form\'e des \'el\'ements dont le stabilisateur est un sous-groupe ouvert de $G$. Si $i\geq 0$ est un entier, la conjecture de Tate pour les cycles de codimension $i$ en cohomologie $\ell$-adique pr\'edit que les applications cycle
	\begin{align}
		CH^i(\cl{X})\otimes_{\Z}{\Q_{\ell}}&\to H^{2i}(\cl{X},\Q_{\ell}(i))^{(1)}, \label{tate2} \\
		CH^i(X)	\otimes_{\Z}{\Q_{\ell}}&\to H^{2i}(\cl{X},\Q_{\ell}(i))^G, \label{tate1} \\
		CH^i(X)\otimes_{\Z}{\Q_{\ell}}&\to H^{2i}(X,\Q_{\ell}(i)) \label{tate3}
	\end{align}
	sont surjectives. En fait ces trois versions de la conjecture de Tate sont \'equivalentes: l'\'equivalence entre la surjectivit\'e des applications (\ref{tate2})	et (\ref{tate1}) suit par un argument de restriction-corestriction, et celle entre la surjectivit\'e des applications (\ref{tate1}) et  (\ref{tate3}) utilise les conjectures de Weil. 
	
	On s'int\'eresse ici aux variantes enti\`eres de la conjecture de Tate. Celles-ci ne sont pas vraies en g\'en\'eral, mais on a des raisons
	d'esp\'erer qu'elles le soient dans le cas $i=d-1$, c'est-\`a-dire pour les $1$-cycles; voir \cite[\S 2]{ctszamuely}. Les questions d'int\'er\^et sont donc les suivantes: les applications 
	\begin{align}
		CH^{d-1}(\cl{X})\otimes_{\Z}{\Z_{\ell}}&\to H^{2d-2}(\cl{X},\Z_{\ell}(d-1))^{(1)},	\label{tate-int2}\tag{1.1'}\\
		CH^{d-1}(X)	\otimes_{\Z}{\Z_{\ell}}&\to H^{2d-2}(\cl{X},\Z_{\ell}(d-1))^G,\label{tate-int1}\tag{1.2'} \\
		CH^{d-1}(X)\otimes_{\Z}{\Z_{\ell}}&\to H^{2d-2}(X,\Z_{\ell}(d-1))	\label{tate-int3}\tag{1.3'}
	\end{align}
	sont-elles surjectives? %Au moment d'\'ecrire ces lignes, la r\'eponse est affirmative dans tous les cas connus. 
	Depuis leur conception, ces questions ont inspir\'e un grand nombre de travaux par de nombreux auteurs. Le but de cet article est de montrer la surjectivit\'e de ces applications (surtout (\ref{tate-int3})) pour certaines vari\'et\'es de dimension $3$.
	
	\subsection{R\'esultats pr\'ec\'edents.} \'Etablir la surjectivit\'e de l'application (\ref{tate-int2}) s'av\`ere \^{e}tre le probl\`eme le plus abordable. Un c\'el\`ebre th\'eor\`eme de Schoen \cite[Theorem (0.5)]{schoen1998integral} montre que si la conjecture de Tate pour les surfaces est vraie (c'est-\`a-dire sous l'hypoth\`ese que l'application (\ref{tate1}) est surjective pour les cycles de dim\'ension $1$ sur toute surface sur  toute extension finie de $\F$) alors l'application (\ref{tate-int2}) est surjective pour tout $X$. Il y a aussi des r\'esultats inconditionnels dans certains cas particuliers: les hypersurfaces cubiques de dimension $4$ par Charles et Pirutka \cite[Th\'eor\`eme 1.1]{charles2015conjecture} et les vari\'et\'es ab\'eliennes de dimension $3$ par Totaro \cite[Theorem 7.1]{totaro}.
	
	Par contre, la surjectivit\'e des applications (\ref{tate-int1}) et (\ref{tate-int3}) est g\'en\'eralement plus difficile \`a montrer que celle de l'application (\ref{tate-int2}). Par exemple, on ne conna\^{i}t pas d'analogues des th\'eor\`emes de Schoen, Charles, Pirutka et Totaro pour ces variantes, et d\'ej\`a pour $X$ un produit de trois courbes elliptiques la surjectivit\'e de l'application (\ref{tate-int3}) est un probl\`eme ouvert bien connu des sp\'ecialistes. Comme l'a not\'e Schoen, la plausibilit\'e de la surjectivit\'e de l'application (\ref{tate-int1}) est une particularit\'e des corps finis. Par exemple, si $\F$ est remplac\'e par $\Q$, alors l'application (\ref{tate-int1}) n'est pas surjective pour $X=Q\times_{\Q} \P^{d-1}_{\Q}$, o\`u $Q\subset \P^2_{\Q}$ est la conique d'\'equation $x_0^2+x_1^2+x_2^2=0$.

	Parmi les trois, la surjectivit\'e de l'application (\ref{tate-int3}) est la plus myst\'erieuse. La surjectivit\'e de l'application (\ref{tate-int3}) entra\^{i}ne celle de l'application (\ref{tate-int1}), mais l'inverse n'est a priori pas vrai: le groupe $H^{2d-2}(X,\Z_{\ell}(d-1))$ contient des classes de cohomologie  ``exotiques'', qui disparaissent apr\`es passage \`a $\cl{\F}$; voir (\ref{hochschild}). L'application (\ref{tate-int3}) est surjective pour $d=1$, et aussi pour $d=2$ par exemple si $b_2(\cl{X})=\rho(\cl{X})$. (Si $V$ est une $\F$-vari\'et\'e, pour tout $j\geq 0$ et tout premier $\ell$ inversible dans $\F$ on note $b_j(\cl{V}):=\dim H^j(\cl{V},\Q_{\ell})$ le $j$-\`eme nombre de Betti de $\cl{V}$ et  $\rho(\cl{V})$ la dimension de l'espace vectoriel $\on{NS}(\cl{V})\otimes_{\Z} \Q_{\ell}$.) La surjectivit\'e de l'application (\ref{tate-int3}) en toute dimension $d\geq 3$ suit de la surjectivit\'e en dimension $d=3$; voir \cite[Proposition 5.4]{ctscavia1}. 
	
	Dans le cas $d=3$, la surjectivit\'e de l'application (\ref{tate-int3}) est li\'ee \`a l'annulation du groupe de cohomologie non-ramifi\'ee $H^3_{\on{nr}}(\F(X),\Q_{\ell}/\Z_{\ell}(2))$. Plus pr\'ecis\'ement, d'apr\`es un th\'eor\`eme de  Colliot-Th\'el\`ene et Kahn \cite[Th\'eor\`eme 2.2, Proposition 3.2]{colliot2013cycles}, si $d=3$ on a l'implication 
	\[H^3_{\on{nr}}(\F(X),\Q_{\ell}/\Z_{\ell}(2))=0\quad  \Longrightarrow\quad \text{(\ref{tate-int3}) est surjective},\]
	et l'inverse vaut si $CH_0(X)_{\Q}$ est support\'e en dimension $2$. (Si $V$ est une $\F$-vari\'et\'e projective et lisse, on dit que $CH_0(V)_{\Q}$ est support\'e en dimension $i$ s'il existe un morphisme $f:W\to V$, o\`u $W$ est une $\F$-vari\'et\'e projective et lisse de dimension $\leq i$, tel que pour tout corps alg\'ebriquement clos $\Omega$ contenant $\F$ l'homomorphisme d'image directe $f_*:CH_0(W_{\Omega})_{\Q}\to CH_0(V_{\Omega})_{\Q}$ est surjectif.) 
	
	Dans \cite[Question 5.4]{colliot2013cycles}, Colliot-Th\'el\`ene et Kahn ont demand\'e:
	\begin{question}\label{question-ctk}
		Est-ce que $H^3_{\on{nr}}(\F(X),\Q_{\ell}/\Z_{\ell}(2))=0$ pour toute $\F$-vari\'et\'e projective et lisse de dimension $3$?
	\end{question}
	Comme nous venons de le discuter, une r\'eponse affirmative \`a la question \ref{question-ctk} impliquerait la surjectivit\'e des applications (\ref{tate-int1}) et (\ref{tate-int3}) en toute dimension. Parimala et Suresh \cite{parimala2016degree} ont donn\'e une r\'eponse affirmative \`a la question \ref{question-ctk} si $X$ est un fibr\'e en coniques au dessus d'une surface. Pirutka \cite{pirutka2016cohomologie} a r\'epondu affirmativement \`a la question \ref{question-ctk} si $X=C\times_{\F}S$ o\`u $C$ est un courbe projective et lisse, $S$ est une surface projective et lisse avec $CH_0(S)_{\Q}$ support\'e en dimension $0$ et le groupe de N\'eron-Severi g\'eom\'etrique  $\on{NS}(\cl{S})$ est sans torsion.

	Dans un pr\'ec\'edent article \cite{ctscavia1}, Colliot-Th\'el\`ene et l'auteur du pr\'esent article ont attaqu\'e la question suivante, cas particuliers du probl\`eme de la surjectivit\'e de l'application (\ref{tate-int3}) et de la question \ref{question-ctk}:	
	\begin{question}\label{question-cts}
		Supposons la caract\'eristique de $\F$ diff\'erente de $2$. Soit $X=C\times_{\F}S$, o\`u $C$ est une courbe elliptique et $S$ est une surface d'Enriques. Est-ce-que \emph{(\ref{tate-int3})} est surjective pour $\ell=2$?	Est-ce-que $H^3_{\on{nr}}(\F(X),\Q_{2}/\Z_{2}(2))=0$?
	\end{question}
	L'int\'er\^et de ce cas particulier est double. D'une part, c'est le cas le plus simple qui \'echappe au th\'eor\`eme de Pirutka, parce que $CH_0(S)_{\Q}$ est support\'e en dimension $0$ mais $\on{NS}(\cl{S})_{\on{tors}}=\Z/2\Z$. D'autre part, sur le corps des complexes, Benoist et Ottem  \cite{benoist2018failure} ont utilis\'e ce genre de vari\'et\'es pour produire de nouveaux contre-exemples \`a la conjecture enti\`ere de Hodge pour les cycles de dimension $1$. Ces contre-exemples sont obtenus par une m\'ethode de sp\'ecialisation, qui n'est pas disponible sur les corps finis. Ils ne peuvent donc pas \^etre facilement adapt\'es pour donner des contre-exemples aux conjectures de Tate enti\`eres sur $\F$ (bien qu'il soit possible de les utiliser pour obtenir des contre-exemples sur des corps de nombres).
	
	Dans \cite{ctscavia1}, nous avons donn\'e dans certains cas une r\'eponse affirmative \`a la question \ref{question-cts} (plus g\'en\'eralement, pour $\ell$ quelconque diff\'erent de la caract\'eristique de $\F$ et dans la situation o\`u $C$ est une courbe de genre $\geq 1$ et $S$ est une surface avec $CH_0(S)_{\Q}$ support\'e en dimension $0$), \emph{sous l'hypoth\`ese que la conjecture de Tate pour les surfaces est vraie}. Il s'agit donc de r\'esultats conditionnels, dans l'esprit du th\'eor\`eme de Schoen.
	Plus pr\'ecis\'ement, soient $C$ et $S$ deux vari\'et\'es g\'eom\'etriquement connexes, projectives et lisses sur $\F$, de dimension $1$ et $2$ respectivement, $J_C$ la jacobienne de la courbe $C$, et $X:=C\times_{\F} S$. Dans \cite{ctscavia1}, nous avons d\'emontr\'e que l'application (\ref{tate-int3}) est surjective et que l'on a $H^3_{\on{nr}}(\F(X),\Q_{\ell}/\Z_{\ell}(2))=0$ sous l'une des hypoth\`eses suivantes:
	\begin{itemize}
		\item le groupe $CH_0(X)_{\Q}$ est support\'e en dimension $0$ et $\ell$ ne divise pas l'ordre de $\on{NS}(\cl{S})_{\on{tors}}$ (\cite[Th\'eor\`eme 1.1]{ctscavia1}). Comme mentionn\'e ci-dessus, le cas particulier $\on{NS}(\cl{S})_{\on{tors}}=0$ avait \'et\'e d\'emontr\'e par Pirutka \cite{pirutka2016cohomologie},
		%\item le groupe $CH_0(X)_{\Q}$ est support\'e en dimension $0$, le groupe ab\'elien fini $\on{Hom}_G(\on{NS}(\cl{S})\set{\ell},J_C(\cl{\F})\set{\ell})$  est nul et l'application cycle
		%$$CH^2(\cl{X})\otimes_{\Z} \Z_{\ell} \to H^4(\cl{X},\Z_{\ell}(2))$$ est surjective (\cite[Th\'eor\`eme 1.3]{ctscavia1}),
		\item la conjecture de Tate pour les surfaces est vraie, $CH_0(X)_{\Q}$ est support\'e en dimension $0$ et le groupe $\on{Hom}_G(\on{NS}(\cl{S})\set{\ell},J_C(\cl{\F})\set{\ell})$ est nul (\cite[Th\'eor\`eme 1.4]{ctscavia1}).
	\end{itemize}
	
	\subsection{R\'esultats principaux.} Notre premier r\'esultat montre l'\'equivalence entre la surjectivit\'e des applications (\ref{tate-int1}) et (\ref{tate-int3}) pour certains produits. 
	
	\begin{thm}\label{mainthm}
		Soient $C$ et $S$ deux vari\'et\'es g\'eom\'etriquement connexes, projectives et lisses sur $\F$, de dimension $1$ et $2$ respectivement et $X:=C\times_{\F} S$. On suppose que l'on a $\rho(\cl{S})=b_2(\cl{S})$. Alors l'application \emph{(\ref{tate-int1})} est surjective si et seulement si l'application \emph{(\ref{tate-int3})} est surjective.
	\end{thm}  
	
	En cons\'equence du th\'eor\`eme \ref{mainthm}, on g\'en\'eralise et rend inconditionnels plusieurs r\'esultats de \cite{ctscavia1}. 
	
	\begin{thm}\label{mainthm'}
		Soient $C$ et $S$ deux vari\'et\'es g\'eom\'etriquement connexes, projectives et lisses sur $\F$, de dimension $1$ et $2$ respectivement et $X:=C\times_{\F} S$. Supposons que l'on ait $\rho(\cl{S})=b_2(\cl{S})$ et notons $J_C$ la jacobienne de $C$ et $(\on{Pic}^0_{S/\F})_{\on{red}}$ la composante connexe du sch\'ema de Picard de $S$, avec sa structure r\'eduite.
\begin{enumerate}[label=(\alph*)]
\item Supposons
\begin{equation}\label{cond2}
{\rm Hom}_{G}(\on{NS}(\cl{S})\{\ell\}, J_C(\cl{\F})\set{\ell})=0\quad\text{et}\quad \on{Hom}_{\F-{\rm gr}}((\on{Pic}^0_{S/\F})_{\on{red}},J_C)=0.
\end{equation}	
Alors l'application \emph{(\ref{tate-int3})} est surjective. 
\item Si la condition \emph{(\ref{cond2})} est satisfaite et $CH_0(S)_{\Q}$ est support\'e en dimension $1$, alors \[H^3_{\on{nr}}(\F(X),\Q_{\ell}/\Z_{\ell}(2))=0.\]
\end{enumerate}
\end{thm}
	
	Si $CH_0(S)_{\Q}$ est support\'e en dimension $0$, alors $b_{1}(\cl{S})=0$, $\rho(\cl{S})=b_{2}(\cl{S})$ et la composante connexe du sch\'ema de Picard de $S$ est triviale. On obtient alors \cite[Th\'eor\`eme 1.4]{ctscavia1} sans supposer la conjecture de Tate pour les surfaces. En cons\'equence, on obtient une r\'eponse affirmative inconditionnelle \`a la question \ref{question-cts} dans certains cas. Dans \cite{ctscavia1}, on avait obtenu l'assertion suivante sous la conjecture de Tate pour les surfaces.
	
	\begin{cor}\label{cor-enriquese}
		Si la caract\'eristique de $\F$ est diff\'erente de $2$ et $X=C\times_{\F}S$, o\`u $S$ est une surface d'Enriques et $C$ est une courbe elliptique sans points d'ordre $2$ d\'efinis sur $\F$, la question \emph{\ref{question-cts}} a r\'eponse affirmative.
	\end{cor}
	
	Si la courbe elliptique $C$ a un mod\`ele affine d'\'equation $y^2=f(x)$, o\`u $f(x)$ est un polyn\^{o}me de degr\'e $3$, la condition du corollaire \ref{cor-enriquese} est remplie si est seulement si $f(x)$ est irr\'eductible.  Ceci est un ph\'enom\`ene typique: pour une classe donn\'ee de courbes $C$ et surfaces $S$, les conditions (\ref{cond2}) sont g\'en\'eralement faciles \`a d\'ecrire explicitement et sont satisfaites par une partie substantielle des membres de la classe.
	
	En combinant le th\'eor\`eme \ref{mainthm'} et la suite exacte de \cite[Th\'eor\`eme 6.3]{colliot2013cycles}
\begin{center}\label{ctkahn}
\begin{tikzcd}[row sep=normal, column sep=small]
0\arrow[r]
&  \on{Ker}\left(CH^2(X)\{\ell\}\to CH^2(\cl{X})\{\ell\}\right) \ar[r] \arrow[d, phantom, ""{coordinate, name=Z}] & H^1\left(\F,  H^3(\cl{X},\Z_{\ell}(2))_{\on{tors}}\right)\arrow[dl,
rounded corners,
to path={ -- ([xshift=2ex]\tikztostart.east)
|- (Z) [near start]\tikztonodes
-| ([xshift=-2ex]\tikztotarget.west)
-- (\tikztotarget)}] &\\
&  \on{Ker}\left(H^3_{\on{nr}}(\F(X),\Q_{\ell}/\Z_{\ell}(2))H^3_{\on{nr}}(\cl{\F}(X),\Q_{\ell}/\Z_{\ell}(2)) \right)\ar[r]  &\on{Coker}\left(CH^2(X)\to CH^2(\cl{X})^G\right) \{\ell\}\ar[r]& 0,
\end{tikzcd}
\end{center}
on obtient la cons\'equence suivante pour la descente galoisienne pour $CH^2(X)$.
	\begin{cor}
		Soient $C$ et $S$ deux vari\'et\'es g\'eom\'etriquement connexes, projectives et lisses sur $\F$, de dimension $1$ et $2$ respectivement et $X:=C\times_{\F} S$. Supposons $\rho(\cl{S})=b_2(\cl{S})$, que la condition \emph{(\ref{cond2})} est satisfaite et que $CH_0(S)_{\Q}$ est support\'e en dimension $\leq 1$. Alors
		\[\on{Ker}\left(CH^2(X)\{\ell\} \to CH^2(\cl{X})\{\ell\}\right)\simeq H^1\left(\F,  H^3(\cl{X},\Z_{\ell}(2))_{\on{tors}}\right)\] et la fl\`eche naturelle $CH^2(X)\to CH^2(\cl{X})^G$ est surjective.
	\end{cor}
	Dans la situation du corollaire \ref{cor-enriquese}, $H^3(\cl{X},\Z_{2}(2))_{\on{tors}}\simeq \Z/2\Z$ et l'application $CH^2(X)\{2\} \to CH^2(\cl{X})\{2\}$ n'est donc jamais injective.

	\subsection{Id\'ees de la d\'emonstration.} Les m\'ethodes utilis\'ees dans cet article diff\`erent sensiblement de celles de \cite{ctscavia1}. Dans \cite{ctscavia1}, sous l'hypoth\`ese que $CH_0(S)_{\Q}$ est support\'e en dimension $0$, le point principal \'etait l'\'etude, par des outils venant de la $K$-th\'eorie alg\'ebrique, du groupe de Chow des  $0$-cycles de degr\'e $0$ sur la surface $S\times_{\F}{\F(C)}$, groupe qui est li\'e \`a $CH^2(X)$ au moyen de la suite de localisation. 
	
	Dans cet article, on utilise une m\'ethode globale. Le r\'esultat cl\'e pour la preuve du th\'eor\`eme \ref{mainthm} est le th\'eor\`eme \ref{iota-dans-image} qui montre que, sous l'hypoth\`ese $b_2(\cl{S})=\rho(\cl{S})$, le noyau de $H^{4}(X,\Z_{\ell}(2)) \to H^{4}(\cl{X},\Z_{\ell}(2))$ est contenu dans l'image de l'application (\ref{tate-int3}). La d\'emonstration du th\'eor\`eme \ref{iota-dans-image} se fait en plusieurs \'etapes. 
	\begin{itemize}
		\item La suite spectrale d'Hochschild-Serre induit un isomorphisme
		\[\theta_X:\on{Ker}(H^{4}(X,\Z_{\ell}(2)) \to H^{4}(\cl{X},\Z_{\ell}(2)))\xrightarrow{\sim} H^1(\F,H^3(\cl{X},\Z_{\ell}(2)));\] voir (\ref{just-for-intro}). Par la formule de K\"unneth $\ell$-adique, le groupe de droite admet une d\'ecomposition en trois facteurs directs
		\[H^1\left(\F, H^3(\cl{S},\Z_{\ell}(2))\right),\qquad  H^1\left(\F,H^1(\cl{S},\Z_{\ell}(1))\right)\]
		et
		\[H^1\left(\F,H^2(\cl{S},\Z_{\ell}(1))\otimes_{\Z_{\ell}} H^1(\cl{C},\Z_{\ell}(1))\right).\]
		En utilisant l'hypoth\`ese $b_2(\cl{S})=\rho(\cl{S})$ et un c\'el\`ebre th\'eor\`eme de Lang \cite{lang1956unramified} (voir aussi \cite[Th\'eor\`eme 5]{colliot1983torsion}), on montre dans les lemmes \ref{lemmeA} et \ref{lemmeB} que les deux premiers termes sont contenus dans l'image de l'application (\ref{tate-int3}).
		\item Pour conclure, il suffit de montrer que le facteur ``mixte'' est contenu dans l'image de l'application (\ref{tate-int3}). On construit un homomorphisme
		\[\on{Pic}(S)\otimes_{\Z}J_C(\F)\otimes_{\Z}\Z_{\ell}\to H^1\left(\F,H^2(\cl{S},\Z_{\ell}(1))\otimes_{\Z_{\ell}} H^1(\cl{C},\Z_{\ell}(1))\right)\]
		en utilisant la th\'eorie de Kummer et le cup-produit en cohomologie galoisienne; voir (\ref{gros}) et (\ref{gros'}). Il n'est pas a priori clair que cet homomorphisme est compatible avec
		\[\on{Pic}(S)\otimes_{\Z} J_C(\F)\otimes_{\Z}\Z_{\ell}\xrightarrow{\cup}CH^2(X)\otimes_{\Z}\Z_{\ell}\xrightarrow{\text{(\ref{tate-int3})}}H^4(X,\Z_{\ell}(2)).\]
		La d\'emonstration de cette compatibilit\'e est le c\oe{}ur technique de ce travail. Pour cela, une \'etude syst\'ematique de la fl\`eche $\theta_X$ est n\'ecessaire. Ceci occupe la plus grande partie des sections \ref{sec3} et \ref{sec4}. Notre analyse s'applique \`a $X$ arbitraire et donc pourrait  potentiellement \^etre utilis\'ee pour des vari\'et\'es plus g\'en\'erales que des produits.
		\item En utilisant cette description et la condition $b_2(\cl{S})=\rho(\cl{S})$, on arrive \`a d\'emontrer que le terme manquant de la d\'ecomposition ci-dessus est dans l'image de l'application (\ref{tate-int3}) si $G$ agit trivialement sur $\on{NS}(\cl{S})$. Par un argument de normes utilisant l'application trace $\ell$-adique, on d\'eduit de cela le th\'eor\`eme \ref{iota-dans-image}. Ce dernier argument, sugg\'er\'e par Jean-Louis Colliot-Th\'el\`ene, n'est pas difficile \`a comprendre mais un peu surprenant: les arguments de restriction-corestriction n'aident g\'en\'eralement pas \`a prouver des cas de la conjecture de Tate enti\`ere.
	\end{itemize}

	\subsection*{Notations}
	Si $A$  est un groupe ab\'elien, $n\geq 1$ est un entier, et $\ell$  est un nombre premier,  on note $A[n]:=\set{a\in A: na=0}$, $A\set{\ell}$ le sous-groupe de torsion $\ell$-primaire de $A$, $A_{\on{tors}}$ le sous-groupe de torsion de $A$, $A/\on{tors}:=A/A_{\on{tors}}$, $T_{\ell}(A)$ le module de Tate de $A$ et $V_{\ell}(A):=T_{\ell}(A)\otimes_{\Z_{\ell}}\Q_{\ell}$. 
	
	Si $k$ est un corps, on note $\cl{k}$ une cl\^{o}ture s\'eparable de $k$. Si $k'/k$ est une extension galoisienne de corps, $G=\on{Gal}(k'/k)$ est le groupe de Galois, et $M$ est un $G$-module continu, on note   $H^i(G,M)$ le $i$-\`eme groupe de cohomologie galoisienne de $G$ \`a valeurs dans $M$, et on \'ecrit $M^G=H^0(G,M)$ pour le sous-module des \'el\'ements fix\'es par $G$. Si $k'=\cl{k}$, on \'ecrit $H^i(k,M)$ pour $H^i(G,M)$. 
	
	Si $k'/k$ est une extension de corps et $X$ est un $k$-sch\'ema, on note $X_{k'}:=X\times_kk'$. Si $k'=\cl{k}$, on \'ecrit $\cl{X}$ pour $X\times_k\cl{k}$. Si $F$ est un faisceau ab\'elien \'etale sur $X$, on notera  $\cl{F}$ l'image r\'eciproque de $F$ le long du morphisme de projection $\cl{X}\to X$. Si $Y\to X$ est un morphisme de sch\'emas, pour tout $n\geq 0$ on \'ecrit \[Y_X^n:=Y\times_XY\times_X\dots\times_XY,\qquad \text{($n$ fois)}.\] 
	Si $k$ est un corps, $\ell$   un nombre premier inversible dans $k$, et $n\geq 1$ un entier, on note  $\mu_{{\ell}^n}$ le faisceau ab\'elien \'etale des racines $\ell^n$-i\`emes de l'unit\'e. Si $i>0$, on note $\mu_{{\ell}^n}^{\otimes i}:=\mu_{{\ell}^n}\otimes\dots\otimes\mu_{{\ell}^n}$ ($i$ fois), et si $i<0$ on note $\mu_{{\ell}^n}^{\otimes i}:=\on{Hom}(\mu_{{\ell}^n}^{\otimes (-i)},\Z/{\ell}^n)$, et $\mu_{{\ell}^n}^{\otimes 0}:= \Z/\ell^n$. 
	
	Soient $k$ un corps et $X$ une $k$-vari\'et\'e, c'est-\`a-dire un $k$-sch\'ema s\'epar\'e de type fini. Si $F$ est un pre-faisceau sur $X$, on note $\check{H}^i(X,F)$ les groupes de cohomologie de \v{C}ech. Si $F$ est un faisceau sur $X$ pour la topologie \'etale sur $X$, on \'ecrit $H^{i}(X,F)$ pour les groupes
	de cohomologie \'etale. On note $H^i\left(X,\Z_{\ell}(j)\right)$ la limite projective des $H^i(X,\mu_{{\ell}^n}^{\otimes j})$ pour $n$ tendant vers l'infini, 
	$H^i(X,\Q_{\ell}(j)):= H^i(X,\Z_{\ell}(j))\otimes_{\Z_{\ell}}{\Q}_{\ell}$ et 	$H^i(X,\Q_{\ell}/\Z_{\ell}(j))$ la limite directe des 
	$H^i(X,\mu_{{\ell}^n}^{\otimes j})$.
	
	On note $\on{Pic}(X)$ le groupe de Picard $H^1_{\on{Zar}}(X,\G_{\on{m}}) \simeq H^1_{\text{\'et}}(X,\G_{\on{m}})$ et  $\on{Br}(X)$ le groupe de Brauer cohomologique $H^2_{\text{\'et}}(X,\G_{\on{m}})$. On note   $CH_i(X)$ le groupe des cycles de dimension $i$ modulo \'equivalence rationnelle et, si $X$ est lisse et irr\'eductible de dimension $d$, on pose $CH^i(X):= CH_{d-i}(X)$. Si $X$ est projective et g\'eom\'etriquement int\`egre, on note $\on{Pic}^0_{X/k}$ la composante connexe du sch\'ema de Picard de $X$ (qui existe d'apr\`es \cite[Theorem 5.4]{kleiman}) et, si $k$ est parfait, on note $(\on{Pic}^0_{X/k})_{\on{red}}$ sa r\'eduction (qui est une vari\'et\'e ab\'elienne si $X$ est g\'eom\'etriquement normale). Si $X$ est propre, lisse et irr\'eductible et $k$ est s\'eparablement clos, on note $\on{NS}(X)$   le groupe de N\'eron-Severi de $X$ (qui est un groupe ab\'elien de type fini), $\rho(X)$ le rang de $\on{NS}(X)$ et $b_i(X):=\on{dim}_{\Q_{\ell}}H^i(X,\Q_{\ell})$ les nombres de Betti de $X$. 	
	
	Si $A$ et $B$ sont deux groupes alg\'ebriques sur le corps $k$, on note $\on{Hom}_{k-\on{gr}}(A,B)$ l'ensemble des homomorphismes de $k$-groupes alg\'ebriques entre $A$ et $B$. Si $A$ et $B$ sont commutatifs, $\on{Hom}_{k-\on{gr}}(A,B)$ admet une structure naturelle de groupe ab\'elien. Si $k$ est s\'eparablement clos, $A$ est une vari\'et\'e ab\'elienne sur $k$ et $\ell$ est un nombre premier inversible dans $k$, on pose $T_{\ell}(A):=T_{\ell}(A(k))$ et $V_{\ell}(A):=T_{\ell}(A)\otimes_{\Z_{\ell}}\Q_{\ell}$.
	
\subsection*{Remerciements}
	Je tiens \`a remercier K{\fontencoding{T1}\selectfont\k{e}}stutis \v{C}esnavi\v{c}ius et le d\'epartement de Math\'ematiques d'Orsay (universit\'e Paris-Saclay) pour leur hospitalit\'e pendant l'ann\'ee 2021. Les commentaires d'un rapporteur sur certaines parties du pr\'ec\'edent article \cite{ctscavia1} ont sugg\'er\'e le pr\'esent travail. Je lui en sais gr\'e. Je remercie aussi Alexei Skorobogatov de m'avoir communiqu\'e une preuve alternative du lemme \ref{hochschild-courbe}, dans le cas o\`u la caract\'eristique de $k$ est nulle et $C$ admet un point $k$-rationnel. Je remercie les rapporteurs du pr\'esent article pour avoir lu attentivement mon manuscrit et pour avoir donn\'e des suggestions utiles.
	
	Je suis tr\`es reconnaissant \`a Jean-Louis Colliot-Th\'{e}l\`ene: il m'a sugg\'er\'e qu'un argument de normes pourrait \^{e}tre utilis\'e pour d\'emontrer le th\'eor\`eme \ref{mainthm}, il m'a aid\'e par de nombreuses discussions et a contribu\'e par beaucoup de commentaires et de corrections \`a cet article.
	
	\section{Pr\'eliminaires}		
	
	\subsection{Passage \`a la limite}	
	
	On rappelle ici des arguments bien connus de passage \`a la limite en alg\`ebre commutative et cohomologie galoisienne, qui seront utilis\'es plusieurs fois dans la suite.
	
	\begin{lemma}\label{tensor-limit}
		Soient $A$ un anneau commutatif avec identit\'e et $(M_n)_{n\in \mathbb{N}}$ un syst\`eme projectif de $A$-modules de longueur finie.
		\begin{enumerate}[label=(\alph*)]
		\item Si $N$ est un $A$-module de pr\'esentation finie, la fl\`eche naturelle  
		\[\varprojlim_n(M_n)\otimes_AN\to \varprojlim_n (M_n\otimes_AN)\] 
		est un isomorphisme de $A$-modules.
		\item Si $N$ est un $\Z$-module de type fini, la fl\`eche naturelle  
		\[\varprojlim_n(M_n)\otimes_{\Z}N\to \varprojlim_n (M_n\otimes_{\Z}N)\] 
		est un isomorphisme de $A$-modules.
		\end{enumerate}
	\end{lemma}
	
	\begin{proof}
		(a) Comme $N$ est de pr\'esentation finie, on dispose d'une suite exacte
		\begin{equation}\label{e.tensor-limit.1}A^r\xrightarrow{\phi'} A^s\xrightarrow{\phi} N\to 0.
		\end{equation}
		Si on tensorise la suite exacte (\ref{e.tensor-limit.1}) par $\varprojlim_n M_n$, on obtient une suite exacte
		\begin{equation}\label{e.tensor-limit.2}(\varprojlim_n M_n)\otimes_AA^r\to (\varprojlim_n M_n)\otimes_AA^s\to (\varprojlim_n M_n)\otimes_A N\to 0.
		\end{equation}
		Pour tout $n\in \mathbb{N}$, soient $K_n:=\on{ker}(\on{id}_{M_n}\otimes\phi)$ et $K'_n:=\on{ker}(\on{id}_{M_n}\otimes\phi')$. Si on tensorise la suite exacte (\ref{e.tensor-limit.1}) par $M_n$, on obtient des suites exactes courtes
		\begin{equation}\label{e.tensor-limit.11}0\to K_n\to M_n\otimes_AA^s\to M_n\otimes_A N\to 0,
		\end{equation}
		\begin{equation}\label{e.tensor-limit.12}0\to K'_n\to M_n\otimes_AA^r\to K_n\to 0.
		\end{equation}
		Comme $K_n\subset M_n\otimes_AA^s$ et la longueur de $M_n\otimes_AA^s\simeq M_n^s$ est finie, par \cite[\S 1.2 p.12]{matsumura} la longueur de $K_n$ est finie. On en d\'eduit que la suite des $K_n$ satisfait la condition de Mittag-Leffler. Le m\^eme argument montre que la suite des $K'_n$ satisfait la condition de Mittag-Leffler. Donc, d'apr\`es \cite[Proposition (2.7.4)]{neukirch2008cohomology} les suites (\ref{e.tensor-limit.11}) et (\ref{e.tensor-limit.12}) restent exactes apr\`es passage \`a la limite projective en $n$. Ceci donne une suite exacte
		\begin{equation}\label{e.tensor-limit.3}\varprojlim_n (M_n\otimes_AA^r)\to \varprojlim_n (M_n\otimes_AA^s)\to \varprojlim_n(M_n\otimes_A N)\to 0,
		\end{equation}
		o\`u les homomorphismes sont induits par $\phi'$ et $\phi$.
		En combinant les suites exactes (\ref{e.tensor-limit.2}) et (\ref{e.tensor-limit.3}), on obtient un diagramme commutatif avec lignes exactes
		\[
		\begin{tikzcd}
			(\varprojlim_n M_n)\otimes_AA^r\arrow[r] \arrow[d,"\wr"] & (\varprojlim_n M_n)\otimes_AA^s \arrow[r] \arrow[d,"\wr"] & (\varprojlim_n M_n)\otimes_A N \arrow[r] \arrow[d] & 0 \\
			\varprojlim_n (M_n\otimes_AA^r)\arrow[r]  &  \varprojlim_n (M_n\otimes_AA^s) \arrow[r]  & \varprojlim_n(M_n\otimes_A N)\arrow[r]  & 0,
		\end{tikzcd}
		\]
		o\`u les fl\`eches verticales sont induites par la propri\'et\'e universelle de la limite inverse. Comme la limite inverse commute avec les sommes directes finies, les fl\`eches verticales de gauche et du milieu sont des isomorphismes. On conclut que la fl\`eche verticale de droite est aussi un isomorphisme, comme voulu.
		
		(b) La structure de $A$-module sur $M_n\otimes_{\Z}N$ est induite par celle de $M_n$. On \'ecrit $N\simeq \Z^r\oplus (\oplus_i\Z/m_i)$. Comme le produit tensoriel commute avec la somme directe, il suffit de montrer que pour tout $m\geq 1$ la fl\`eche
		\[\varprojlim_n(M_n)\otimes_{\Z}\Z/m\to \varprojlim_n (M_n\otimes_{\Z}\Z/m)\]
		est un isomorphisme. Ceci suit du fait que pour tout $A$-module $M$, $M\otimes_\Z \Z/m\simeq M\otimes_AA/mA$, du fait que $A/mA$ est un $A$-module de pr\'esentation finie et de la partie (a).
	\end{proof}
	
	\begin{lemma}\label{lim-tensor-lim}
		Soient $A$ un anneau commutatif avec identit\'e, $(M_n)_{n\in \mathbb{N}}$ et $(N_n)_{n\in \mathbb{N}}$ deux syst\`emes projectifs de $A$-modules de longueur finie tels que $M:=\varprojlim_n M_n$ et $N:=\varprojlim_n N_n$ sont des
		$A$-modules de pr\'esentation finie. Alors la fl\`eche naturelle
		\[M\otimes_AN \to \varprojlim_n (M_n\otimes_AN_n)\]
		est un isomorphisme de $A$-modules. 
	\end{lemma}
	
	\begin{proof}
		On a une cha\^{\i}ne d'isomorphismes de $A$-modules:	
\[
			M\otimes_AN\stackrel{\sim}{\longrightarrow} \varprojlim_m(M_m\otimes_A N)
			\stackrel{\sim}{\longrightarrow}\varprojlim_m(\varprojlim_n(M_m\otimes_A N_n)
			\stackrel{\sim}{\longrightarrow} \varprojlim_{(m,n)} (M_m\otimes_A N_n) \stackrel{\sim}{\longrightarrow}  \varprojlim_{n} (M_n\otimes_A N_n).\]
		Le premier et le deuxi\`eme isomorphisme viennent du lemme \ref{tensor-limit}, le troisi\`eme est un cas particulier de la formule pour la limite sur une cat\'egorie produit et le dernier suit du fait que la suite des $(n,n)$ pour $n\geq 1$ est cofinale dans $\mathbb{N}\times \mathbb{N}$. 
	\end{proof}

	\begin{lemma}\label{limit-coho}
		Soient $G=\hat{\Z}$ et $\sigma\in G$ un g\'en\'erateur topologique.
		\begin{enumerate}[label=(\alph*)]
		\item Si $M$ est un $G$-module continu fini, alors 
		\begin{align*}
			H^i(G,M)\simeq 
			\begin{cases}
				M^G, &i=0,\\
				M/(1-\sigma)M, &i=1,\\
				0, &i\geq 2.
			\end{cases}	
		\end{align*}
		En particulier, $H^i(G,M)$ est fini pour tout $i\geq 0$.
		\item Soit $(M_n)_{n\in \mathbb{N}}$ un syst\`eme projectif de $\Z_{\ell}$-modules finis avec action continue de $G$, tel que le $\Z_{\ell}$-module $M:=\varprojlim_n M_n$ est de type fini. Alors pour tout $i\geq 0$ la fl\`eche naturelle
		\[H^i(G,M)\longrightarrow \varprojlim_n H^i(G,M_n)\]
		est un isomorphisme.
		\end{enumerate}
	\end{lemma}
	
	\begin{proof}
		(a) L'assertion $H^0(G,M)=M^G$ est \'evidente, et $H^i(G,M)=0$ pour tout $i\geq 2$ parce que $\on{cd}(G)=1$. On a un homomorphisme $\phi:H^1(G,M)\to M/(1-\sigma)M$ d\'efini comme suit: si $\alpha\in H^1(G,M)$ est repr\'esent\'e par le cocycle $\set{\alpha_g}_{g\in G}$ \`a valeurs dans $M$, on pose $\phi(\alpha):=\alpha_{\sigma}+ (1-\sigma)M$. On v\'erifie ais\'ement qui $\phi$ est bien d\'efini et un isomorphisme de groupes.
		
		(b) Ceci suit de (a) et de \cite[Corollary (2.7.6)]{neukirch2008cohomology}.
	\end{proof}

	\subsection{Application trace} 
	Soient $f:X'\to X$ un morphisme fini \'etale de sch\'emas et $F$ un faisceau ab\'elien \'etale sur $X$. On \'ecrit \[\on{tr}_f:f_*f^*F\to F\] pour indiquer le morphisme trace, d\'efini dans \cite[IX, 5.1]{sga4III}. Il est univoquement caract\'eris\'e par la compatibilit\'e avec tout changement de base \'etale et par le fait que si $X'=\amalg_{j=1}^dX$ et $f$ est l'identit\'e sur chaque terme, alors l'application induite ${f}_*{f}^*F\simeq F^d\to F$ est la somme. 
	
	Comme $f$ est fini, $f_*$ est exact. On a alors pour tout $i\geq 0$ un isomorphisme $H^i(X,f_*f^*F)\simeq H^i(X',f^*F)$, et donc un homomorphisme induit
	\[\on{tr}_f:H^i(X',f^*F)\to H^i(X,F).\]
	
	Une d\'efinition \'equivalente de $\on{tr}_f$ est donn\'ee dans \cite[XVII, Th\'eor\`eme 6.2.3]{sga4III}: comme $f$ est fini \'etale, la fl\`eche naturelle $f_!\to f_*$ est un isomorphisme, et alors $\on{tr}_f$ est induit par l'homomorphisme d'adjonction $f_*f^*\simeq f_!f^*\to \on{id}$.
	
	D'apr\`es \cite[XVII, Th\'eor\`eme 6.2.3 (Var 1) et (Var 2)]{sga4III},  $\on{tr}_f$ est covariant en $F$ et commute avec tout changement de base.
	
	Soient $\ell$ un nombre premier inversible dans $k$ et $j\in \Z$. La fonctorialit\'e de $\on{tr}_f$ nous donne, pour tout $n\geq 0$, un carr\'e commutatif
	\[
	\begin{tikzcd}
		H^i\left(X',\mu^{\otimes j}_{\ell^{n+1}}\right) \arrow[r, "\on{tr}_f"]  \arrow[d] & H^i\left(X,\mu^{\otimes j}_{\ell^{n+1}}\right) \arrow[d] \\
		H^i\left(X',\mu^{\otimes j}_{\ell^n}\right) \arrow[r, "\on{tr}_f"] & H^i\left(X,\mu^{\otimes j}_{\ell^n}\right),
	\end{tikzcd}
	\]
	les fl\`eches verticales \'etant induites par l'homomorphisme de puissance $\ell$-i\`eme. Par passage \`a la limite projective en $n$, on obtient un homomorphisme
	\[\on{tr}_f: H^i(X',\Z_{\ell}(j))\to H^i(X,\Z_{\ell}(j)).\]

	\subsection{Application cycle} 
	Soient $k$ un corps, $X$ une $k$-vari\'et\'e lisse \'equidimension\-nelle et $n$ un entier inversible dans $k$. %et $D^+(X,\Z/n)$ la cat\'egorie d\'eriv\'ee born\'ee inf\'erieurement de la cat\'egorie ab\'elienne des faisceaux de $\Z/n$-modules sur le petit site \'etale de $X$. 
	Pour tout $i\geq 0$, on note
	\[\on{cl}_X:CH^i(X)\to H^{2i}\left(X,\mu_{n}^{\otimes i}\right)\]
	%CH^i(X)\xrightarrow{\on{cl}^h_X} H_{2d-2i}(X,i,n),\qquad
	%l'application cycle en homologie \'etale, d\'efinie dans \cite[Chapitre 4, \S 2.3.1]{SGA4.5}, 
	l'application cycle en cohomologie \'etale, d\'efinie de fa\c{c}on cohomologique dans \cite[Chapitre 4, \S 2.2.10]{SGA4.5}, ou en termes d'homologie et dualit\'e de Poincar\'e dans le paragraphe apr\`es \cite[Chapitre 4, (2.3.1.3)]{SGA4.5}.
	%\[\on{cl}_X=\rho_{X}\circ \on{cl}^h_X.\]
	Comme on peut le v\'erifier \`a l'aide de la d\'efinition cohomologique, les applications $\on{cl}_X:CH^i(X)\to H^{2i}(X,\mu_{n}^{\otimes i})$ forment un syst\`eme inverse de fl\`eches compatibles pour $n$ variable. Si $\ell$ est un nombre premier inversible dans $k$, on notera
	\[\on{cl}_X:CH^i(X)\to H^{2i}(X,\Z_{\ell}(i))\]
	l'homomorphisme obtenu par passage \`a la limite sur $n=\ell^m$, $m\geq 0$.
	\begin{lemma}\label{compat}
		Soit $f:X'\to X$ un morphisme fini \'etale  de $k$-vari\'et\'es lisses \'equidimensionnelles. Pour tout $i\geq 0$ et tout nombre premier $\ell$ inversible dans $k$, le diagramme
		\[
		\begin{tikzcd}
			CH^i(X') \arrow[r,"\on{cl}_{X'}"] \arrow[d,"f_*"] &  H^{2i}(X',\Z_{\ell}(i)) \arrow[d,"\on{tr}_f"] \\
			CH^i(X) \arrow[r,"\on{cl}_{X}"] & H^{2i}(X,\Z_{\ell}(i))
		\end{tikzcd}
		\]	
		commute.
	\end{lemma}

	\begin{proof}
		Pour tout entier $n\geq 1$ inversible dans $k$, on a un diagramme commutatif
		\begin{equation}\label{compat-fini}
			\begin{tikzcd}
				CH^i(X') \arrow[r,"\on{cl}^h_{X'}"] \arrow[d,"f_*"] &  H_{2i}(X',i,n) \arrow[d,"f_*"] \arrow[r,"\sim"]  & H^{2i}\left(X',\mu_n^{\otimes i}\right)  \arrow[d,"\on{tr}_f"]  \\
				CH^i(X) \arrow[r,"\on{cl}^h_{X}"] & H_{2i}(X,i,n) \arrow[r,"\sim"] & H^{2i}\left(X,\mu_n^{\otimes i}\right).
			\end{tikzcd}
		\end{equation}
		Ici $H_{2i}(X',i,n)$ et $H_{2i}(X,i,n)$ sont les groupes d'homologie \'etale et les isomorphismes horizontaux de droite proviennent de la dualit\'e de Poincar\'e; voir \cite[Example (2.1)]{bloch-ogus}. Dans (\ref{compat-fini}), la commutativit\'e du carr\'e de droite est une cons\'equence de \cite[XVIII, Lemme 3.2.3]{sga4III}. La commutativit\'e du carr\'e de gauche suit de la fonctorialit\'e de l'homologie \'etale par rapport aux morphismes propres \cite[p. 185 ligne 20]{bloch-ogus} et du fait que, si $Y'\subset X'$ est un sous-sch\'ema ferm\'e int\`egre de codimension $i$, $Y:=f(Y')$ avec sa structure r\'eduite, $\eta_Y$ et $\eta_{Y'}$ sont les classes fondamentales homologiques de $Y$ et $Y'$ (voir \cite[(1.3.4) et p. 186 ligne -2]{bloch-ogus}) et $g:=f|_{Y'}:Y'\to Y$, alors $g_*(\eta_{Y'})=\on{deg}(g)\cdot\eta_{Y}$ (voir \cite[(7.1.2)]{bloch-ogus}) et $f_*([Y'])=\on{deg}(g)\cdot [Y]$. 
		
		Les fl\`eches horizontales compos\'ees dans (\ref{compat-fini}) sont $\on{cl}_{X'}$ et $\on{cl}_X$. On conclut par passage \`a la limite projective dans (\ref{compat-fini}) en $n=\ell^m$, $m\geq 1$.
	\end{proof}

	\subsection{Hochschild-Serre}\label{hoch-serre-subsec}
	Soient  $\mc{A}$, $\mc{B}$ et $\mc{C}$ des cat\'egories ab\'eliennes et $F:\mc{A}\to \mc{B}$ et $G:\mc{B}\to\mc{C}$ des foncteurs exacts \`a gauche. Supposons que $\mc{A}$ et $\mc{B}$ ont suffisamment d'injectifs et que $F(I)$ est $G$-acyclique pour tout injectif $I$ de $\mc{A}$. Alors pour tout objet $A$ de $\mc{A}$ on a la suite spectrale de Grothendieck (voir \cite[015N]{stacks-project}):
	\begin{equation}\label{groth}
		E_{2}^{p,q}=(R^pG)(R^qF)(A)\Rightarrow R^{p+q}(G\circ F)(A).
	\end{equation}
	
	Soient $k$ un corps, $X$ une $k$-vari\'et\'e et $F$ un faisceau ab\'elien \'etale sur $X$. 
	On a la suite spectrale d'Hochschild-Serre:
	\begin{equation}\label{hochschild-serre-ss}
		E_2^{p,q}=H^p\left(k, H^q(\cl{X},\cl{F})\right)\Rightarrow H^{p+q}(X,F).
	\end{equation}
	On rappelle ici la construction de la suite spectrale (\ref{hochschild-serre-ss}) donn\'ee dans \cite[VIII, Proposition 8.4]{sga4II}; elle nous sera utile dans la preuve du lemme \ref{cech-derived}. Pour toute extension galoisienne finie $k'/k$, la projection $X_{k'}\to X$ est un morphisme couvrant dans le petit site \'etale sur $X$. Par \cite[V, Corollaire 3.4]{sga4I}, on a la suite spectrale de Cartan-Leray (cas particulier de la suite spectrale (\ref{groth})): 
	\begin{equation}\label{cartan-leray}
		E_2^{p,q}=\check{H}^p\left(X_{k'}/X,\mc{H}^q(X,F)\right)\Rightarrow H^{p+q}(X,F).
	\end{equation}
	Ici $\mc{H}^q(X,F)$ est le pr\'efaisceau \'etale sur $X$ qui \`a tout morphisme \'etale de type fini $V\to X$ associe $H^q(V,F|_V)$. Un calcul standard montre que le complexe de \v{C}ech du faisceau $\mc{H}^q(X,F)$ associ\'e au morphisme $X_{k'}\to X$ est isomorphe au complexe de cocha\^{i}nes homog\`enes de $\on{Gal}(k'/k)$ dans $H^q(X_{k'},F|_{X_{k'}})$. Cet isomorphisme est rappel\'e en d\'etail dans \cite[Example III.2.6]{milne1980etale}. On obtient des isomorphismes 
	\begin{equation}\label{cech-galois-isom}
	\check{H}^p\left(X_{k'}/X,\mc{H}^q(X,F)\right)\simeq H^p\left(\on{Gal}(k'/k),H^q(X_{k'},F|_{X_{k'}})\right)
	\end{equation} pour tout $p,q\geq 0$. Par cons\'equent la suite spectrale (\ref{cartan-leray}) induit une suite spectrale
	\begin{equation}\label{hochschild-serre-fini}
		E_2^{p,q}=H^p\left(\on{Gal}(k'/k),H^q(X_{k'},F|_{X_{k'}})\right)\Rightarrow H^{p+q}(X,F).	
	\end{equation}
	La suite spectrale (\ref{hochschild-serre-ss}) est obtenue comme limite inductive des suite spectrales (\ref{hochschild-serre-fini}), o\`u $k'/k$ parcourt l'ensemble des sous-extensions galoisiennes finies de $\cl{k}/k$.
	
	Une deuxi\`eme construction	de la suite spectrale (\ref{hochschild-serre-ss})  comme cas particulier de la suite spectrale (\ref{groth}) est mentionn\'ee \`a la fin de la preuve de \cite[VIII, Proposition 8.4]{sga4II}.  L'\'equivalence des deux constructions est implicite dans \cite[VIII, Proposition 8.4]{sga4II} et n'est pas difficile \`a d\'emontrer. %Dans cet article on utilisera seulement la premi\`ere construction. et donn\'ee dans  \cite[Remark III.2.21(b)]{milne1980etale}).
	
	Le lemme suivant est responsable du signe dans l'\'enonc\'e du lemme \ref{cech-derived}.
	
	\begin{lemma}\label{signe}
		Dans la situation de la suite spectrale \emph{(\ref{groth})}, soit 
		\begin{equation}\label{courte-abelian}
			0\longrightarrow A\longrightarrow B\longrightarrow C\longrightarrow 0
		\end{equation} 
		une suite exacte courte dans $\mc{A}$. Alors pour tout entier $i\geq 1$ on a un carr\'e anticommutatif
		\[
		\begin{tikzcd}
			\on{Ker}\left(R^{i}(G\circ F)(C)\to G(R^{i-1}F)(C)\right) \arrow[r] \arrow[d]  & (R^1G)(R^{i-1}F)(C) \arrow[d] \\
			\on{Ker}\left(R^{i+1}(G\circ F)(A)\to G(R^iF)(A)\right) \arrow[r] & (R^1G)(R^iF)(A).	
		\end{tikzcd}
		\]
		Ici les fl\`eches verticales sont induites par les homomorphismes de connexion dans la suite exacte longue associ\'ee \`a la suite \emph{(\ref{courte-abelian})}, la fl\`eche horizontale du bas est la compos\'ee
		\[\on{Ker}\left(R^{i+1}(G\circ F)(A)\to G(R^iF)(A)\right)\relbar\joinrel\twoheadrightarrow E_{\infty}^{1,i}\subset (R^1G)(R^iF)(A)\]
		induite par \emph{(\ref{groth})} et celle du haut est d\'efinie de la m\^eme mani\`ere.
	\end{lemma}
	
	\begin{proof}
		Il existe un diagramme commutatif 
		\[
		\begin{tikzcd}
			0 \arrow[r] & I^{\bullet}_A \arrow[r] & I^{\bullet}_B \arrow[r] & I^{\bullet}_C \arrow[r] & 0\\ 
			0 \arrow[r] & A \arrow[r] \arrow[u] & B \arrow[r] \arrow[u] & C \arrow[r] \arrow[u] & 0 
		\end{tikzcd}
		\]
		o\`u les lignes horizontales sont exactes et les homomorphismes verticales sont r\'esolutions injectives dans $\mc{A}$.
		Soit encore
		\[
		\begin{tikzcd}
			0 \arrow[r] & J^{\bullet,\bullet}_A \arrow[r, "\iota^{\bullet,\bullet}"] & J^{\bullet,\bullet}_B \arrow[r, "\pi^{\bullet,\bullet}"] & J^{\bullet,\bullet}_C \arrow[r] & 0 \\
			0 \arrow[r] & F(I^{\bullet}_A) \arrow[r] \arrow[u] & F(I^{\bullet}_B) \arrow[r] \arrow[u] & F(I^{\bullet}_C) \arrow[r] \arrow[u] & 0 
		\end{tikzcd}
		\]
		un diagramme commutatif o\`u les lignes horizontales sont exactes et les homomorphismes verticales sont r\'esolutions de Cartan-Eilenberg dans $\mc{B}$; voir \cite[015G]{stacks-project}. 
		
		On suit les conventions de signe de \cite[010V, OFNB]{stacks-project}. Donc $J_A^{\bullet,\bullet}$ est un complexe double concentr\'e dans le quadrant $p,q\geq 0$, avec des homomorphismes $d_1^{p,q}:J_A^{p,q}\to J_A^{p+1,q}$ et $d_2^{p,q}:J_A^{p,q}\to J_A^{p,q+1}$ satisfaisant \[d_1^{p+1,q}\circ d_1^{p,q}=0,\quad d_2^{p,q+1}\circ d_2^{p,q}=0,\quad d_2^{p,q+1}\circ d_1^{p,q}=d_1^{p+1,q}\circ d_2^{p,q}.\]
		Le complexe total associ\'e est $(\on{Tot}^{\bullet}(J_A^{\bullet,\bullet}),d)$, o\`u \[\on{Tot}^n(J_A^{\bullet,\bullet}):=\bigoplus_{p+q=n}J_A^{p,q},\qquad d^n:=\sum_{p+q=n}(d_1^{p,q}+(-1)^pd_2^{p,q}).\]
		La m\^eme discussion s'applique apr\`es avoir remplac\'e $A$ par $B$ ou $C$.
		
		Comme $J_A^{\bullet,\bullet}$, $J_B^{\bullet,\bullet}$ et $J_C^{\bullet,\bullet}$ sont des r\'esolutions de Cartan-Eilenberg, pour tout $q\geq 0$ on peut trouver des homomorphismes de complexes
		\[s^{q}=(s^{p,q})_{p\geq 0}:J_C^{\bullet, q}\to J_B^{\bullet,q},\qquad r^{q}=(r^{p,q})_{p\geq 0}:J_B^{\bullet,q}\to J_A^{\bullet,q}\] satisfaisant $\pi^{\bullet,q}\circ s^{q}=\on{id}_{J_C^{\bullet, q}}$ et $r^{q}\circ\iota^{\bullet,q}=\on{id}_{J_A^{\bullet, q}}$. Soient $(J^{\bullet,\bullet}_A[0,1],d_1[0,1],d_2[0,1])$ le complexe double d\'efini par \[J^{\bullet,\bullet}_A[0,1]^{p,q}=J^{p,q+1}_A,\qquad d_1[0,1]=d_1,\quad d_2[0,1]=-d_2\] et $\delta:J^{\bullet,\bullet}_C\to J^{\bullet,\bullet}_A[0,1]$ le morphisme de complexes doubles donn\'e par $\delta^n:=r^{n+1}\circ d_{J_B^{\bullet,\bullet}}^n\circ s^n$. Le fait que $\delta$ est un morphisme suit d'un calcul direct; voir \cite[011J]{stacks-project}. 
		
		Soit encore \[0\longrightarrow \on{Tot}(J^{\bullet,\bullet}_A)\xrightarrow{\on{Tot}(\iota^{\bullet,\bullet})} \on{Tot}(J^{\bullet,\bullet}_B)\xrightarrow{\on{Tot}(\pi^{\bullet,\bullet})} \on{Tot}(J^{\bullet,\bullet}_C)\longrightarrow 0\] la suite exacte courte induite au niveau des complexes totaux, 	\[(s')^n:=\sum_{p+q=n}s^{p,q},\qquad (r')^n:=\sum_{p+q=n}r^{p,q}\] et $\delta': \on{Tot}(J^{\bullet,\bullet}_C)\to \on{Tot}(J^{\bullet,\bullet}_A)[1]$ le morphisme de complexes $(\delta ')^ n := (r')^{n + 1} \circ d_{\text{Tot}(J_B^{\bullet, \bullet })}^ n \circ (s')^ n$; voir encore \cite[011J]{stacks-project}. Ici, pour un complexe $(K,d_K)$, le complexe $(K[1],d_K[1])$ est d\'efini par $(K[1])^n=K^{n+1}$ et $(d_K[1])^n=-d^{n+1}_K$. 
		
		Les homomorphismes $\delta$ et $\delta'$ sont uniquement d\'efinis \`a homotopie pr\`es par \cite[011L]{stacks-project} et induisent les homomorphismes de bord provenant du lemme du serpent par \cite[011K]{stacks-project}. Pour tout $p,q\geq 0$, soit $(\delta')^{p,q}$ la restriction de $(\delta')^{p+q}$ \`a $J_C^{p,q}$.  Par un calcul direct, explicit\'e dans la quatri\`eme partie de \cite[0G6A]{stacks-project}, on a \[(\delta')^{p,q}=(-1)^p\delta^{p,q}\] pour tout $p,q\geq 0$. En particulier, $(\delta')^{1,i-1}=-\delta^{1,i-1}$. 
		
		Par application de $G$ on obtient une suite exacte courte de complexes doubles
		\[0\longrightarrow   G(J^{\bullet,\bullet}_A)\xrightarrow{G(\iota)}  G(J^{\bullet,\bullet}_B) \xrightarrow{G(\pi)}  G(J^{\bullet,\bullet}_C) \longrightarrow  0. \]
		La conclusion suit du fait que la suite spectrale (\ref{groth}) pour $A$, $B$ et $C$ est d\'efinie comme la deuxi\`eme suite spectrale pour $G(J^{\bullet,\bullet}_A)$, $G(J^{\bullet,\bullet}_B)$ et $G(J^{\bullet,\bullet}_C)$, respectivement (voir \cite[015N]{stacks-project}), et du fait que dans le carr\'e de l'\'enonc\'e du lemme, la fl\`eche de gauche est induite par $G((\delta')^{1,i-1})$ et celle de droite par $G(\delta^{1,i-1})$.
	\end{proof}
	
	\subsection{Cohomologie de \v{C}ech} Soient $k$ un corps et $X$ une $k$-vari\'et\'e quasi-projective. D'apr\`es un th\'eor\`eme d\^u \`a Artin \cite[Corollary 4.2]{artin1971joins} (voir aussi \cite[Theorem III.2.17]{milne1980etale}), pour tout faisceau ab\'elien \'etale $F$ sur $X$ et pour tout $i\geq 0$ on a un isomorphisme fonctoriel
	\begin{equation}\label{derived-cech}
		\check{H}^i(X,F) \stackrel{\sim}{\longrightarrow} H^i(X,F).
	\end{equation} 
	Cet isomorphisme est le morphisme de coin dans la suite spectrale de Cartan-Leray
	\[E_2^{p,q}=\check{H}^p(X,\mc{H}^q(X,F))\Longrightarrow H^{p+q}(X,F),\]
	et Artin montre que l'on a $E^{p,q}_2=0$ pour tout $p\geq 0$ et $q\geq 1$. En particulier, les isomorphismes (\ref{derived-cech}) sont compatibles avec les morphismes de coin dans la suite spectrale (\ref{cartan-leray}).
	
	Pour toute suite exacte courte de faisceaux ab\'eliens \'etales \[0\longrightarrow F'\longrightarrow F\longrightarrow F''\longrightarrow 0,\] on d\'eduit un diagramme commutatif \`a lignes exactes
	\begin{equation}\label{derived-cech2}
		\adjustbox{scale=0.95,center}{ 
			\begin{tikzcd}
				\cdots\arrow[r] & \check{H}^i(X,F') \arrow[r] \arrow[d,"\wr"] & \check{H}^i(X,F) \arrow[r] \arrow[d,"\wr"]  & \check{H}^i(X,F'') \arrow[r] \arrow[d,"\wr"]   & \check{H}^{i+1}(X,F') \arrow[d,"\wr"]\arrow[r] & \cdots \\
				\cdots\arrow[r] & H^i(X,F') \arrow[r]  & H^i(X,F) \arrow[r]  & H^i(X,F'') \arrow[r] & H^{i+1}(X,F')  \arrow[r] & \cdots,
			\end{tikzcd}
		}
	\end{equation}
	o\`u les homomorphismes de bord dans la suite exacte du haut sont induits par les cobords de \v{C}ech usuels; \emph{cf.} \cite[\S 5.11]{godement}.

	\section{L'application $\theta_X$}\label{sec3}	
	Soient $k$ un corps, $X$ une $k$-vari\'et\'e et $F$ un faisceau ab\'elien \'etale sur $X$. La suite spectrale (\ref{hochschild-serre-ss}) donne, pour tout $i\geq 0$, des homomorphismes
	\begin{equation}\label{hochschild-derived}
	\theta_X:\on{Ker}\left({H}^{i+1}(X,F)\to {H}^{i+1}(\cl{X},\cl{F})\right)\relbar\joinrel\twoheadrightarrow E_{\infty}^{1,i}\subset H^1\left(k, H^{i}(\cl{X},\cl{F})\right)
	\end{equation}
	fonctoriels en $X$ et $F$. Comme la suite spectrale (\ref{hochschild-serre-ss}) est la limite inductive des suites spectrales (\ref{hochschild-serre-fini}), $\theta_X$ est la limite inductive des homomorphismes
	\[\theta_{X,k'}:\on{Ker}\left(H^1(X,F)\to H^1(X_{k'},F|_{X_{k'}})\right)\relbar\joinrel\twoheadrightarrow E_{\infty}^{1,i}\subset H^1\left(\on{Gal}(k'/k),H^1(X_{k'},F|_{X_{k'}})\right)\]
	provenant de la suite spectrale (\ref{hochschild-serre-fini}), o\`u $k'/k$ parcourt l'ensemble des sous-extensions galoisiennes finies de $\cl{k}/k$. 
	
	On veut donner une description en cohomologie de \v{C}ech de l'homomorphisme (\ref{hochschild-derived}). Pour tout entier $i\geq 0$, consid\'erons le diagramme suivant:
	\begin{equation}\label{hochschild-cech}
		\begin{tikzcd}
			\on{Ker}\left(\check{H}^{i+1}(X,F)\to \check{H}^{i+1}(\cl{X},\cl{F})\right) \arrow[r,"\check{\theta}_X"] \arrow[d,"\wr"] & H^1\left(k,\check{H}^i(\cl{X},\cl{F})\right) \arrow[d,"\wr"] \\
			\on{Ker}\left({H}^{i+1}(X,F)\to {H}^{i+1}(\cl{X},\cl{F})\right) \arrow[r,"\theta_X"]  & H^1\left(k,{H}^i(\cl{X},\cl{F})\right).   
		\end{tikzcd}
	\end{equation}
	Ici les fl\`eches verticales sont induites par les isomorphismes (\ref{derived-cech}) et $\check{\theta}_X$ est d\'efinie comme suit.  Soient $\alpha\in \check{H}^{i+1}(X,F)$ qui devient nul dans $\check{H}^{i+1}(\cl{X},\cl{F})$, $U\to X$ un morphisme \'etale surjectif de type fini tel que $\alpha$ est repr\'esent\'e par $\beta\in F(U^{i+2}_X)$ et $\gamma\in \cl{F}(\cl{U}^{i+1}_{\cl{X}})$ tel que l'image $\cl{\beta}$ de $\beta$ dans $\cl{F}(\cl{U}^{i+2}_{\cl{X}})$ est \'egale \`a $d(\gamma)$, o\`u $d$ est le cobord de \v{C}ech. Comme $\cl{\beta}$ est d\'efini sur $k$, il est $G$-invariant, et donc pour tout $g\in G$ on a \[d(g(\gamma)-\gamma)=g(d(\gamma))-d(\gamma)=g(\beta)-\beta=0 \text{ dans $\cl{F}(\cl{U}^{i+1}_{\cl{X}})$.}\] Donc $g(\gamma)-\gamma$ est un cocycle de \v{C}ech pour tout $g\in G$. Si $c$ est un cocycle (de \v{C}ech ou galoisien), on note $[c]$ la classe de cohomologie associ\'ee. Pour tout $g,h\in G$, on a
	\[(gh)(\gamma)-\gamma=g(\gamma)-\gamma+g(h(\gamma)-\gamma),\] 
	donc $\set{[g(\gamma)-\gamma]}_{g\in G}$ est un cocycle galoisien. On d\'efinit \begin{equation}\label{def-cech-theta}\check{\theta}_X(\alpha):=\left[\big\{[g(\gamma) -\gamma]\big\}_{g\in G}\right]\in H^1\left(k, \check{H}^i(\cl{X},\cl{F})\right).
	\end{equation}
	On peut ais\'ement v\'erifier que cette d\'efinition ne d\'epend que de $\alpha$. De la m\^eme mani\`ere, pour toute extension galoisienne finie $k'/k$ on peut d\'efinir un homomorphisme 
	\[\check{\theta}_{X,k'}:\on{Ker}\left(\check{H}^{i+1}(X,F)\to \check{H}^{i+1}(X_{k'},F|_{X_{k'}})\right)\longrightarrow H^1\left(\on{Gal}(k'/k),H^i(X_{k'},F|_{X_{k'}})\right)\] 
	en rempla\c{c}ant partout $\cl{k}$ par $k'$ dans (\ref{def-cech-theta}). Il est imm\'ediat de v\'erifier que $\check{\theta}_{X}$ est la limite inductive des $\check{\theta}_{X,k'}$, o\`u $k'/k$ parcourt l'ensemble des sous-extensions galoisiennes finies de $\cl{k}/k$.

	\begin{lemma}\label{cech-derived}
		Soient $X$ une $k$-vari\'et\'e quasi-projective, $F$ un faisceau ab\'elien \'etale sur $X$ et $i\geq 0$ un entier. Le carr\'e \emph{(\ref{hochschild-cech})} $(-1)^i$-commute.
	\end{lemma}
	
	\begin{proof}
		On pose  	
		\begin{align*}
			\check{H}^{i+1}_0(X,F)&:=\on{Ker}\left(\check{H}^{i+1}(X,F)\longrightarrow \check{H}^{i+1}(\cl{X},\cl{F})\right), \\
			H^{i+1}_0(X,F)&:=\on{Ker}\left({H}^{i+1}(X,F)\longrightarrow {H}^{i+1}(\cl{X},\cl{F})\right).
		\end{align*}	
		Soit \begin{equation}\label{eff'}
		0\longrightarrow F\longrightarrow I\longrightarrow F'\longrightarrow 0
		\end{equation}
		une suite exacte courte de faisceaux ab\'eliens \'etales sur $X$, o\`u $I$ est donn\'e par la construction de Godement pour $F$; voir \cite[Remark III.1.20(c)]{milne1980etale}. Le faisceau $I$ est un produit de faisceaux gratte-ciel \'etales sur $X$; donc $I$ est flasque. Il s'en suit que $\cl{I}$ est un produit de faisceaux gratte-ciel \'etales sur $\cl{X}$ et donc $\cl{I}$ est aussi flasque. Le foncteur d'image r\'eciproque \'etant exact, la suite (\ref{eff'}) induit une suite exacte courte 
		\begin{equation}\label{eff'cl}
		0\longrightarrow \cl{F}\longrightarrow \cl{I}\longrightarrow \cl{F'}\longrightarrow 0
		\end{equation}
		de faisceaux ab\'eliens \'etales sur $\cl{X}$.
		Consid\'erons le diagramme
		\[
		\begin{tikzcd}
			\check{H}^{i+1}_0(X,F) \arrow[ddd] \arrow[rrr]  &&& H^1\left(k,\check{H}^i(\cl{X},\cl{F})\right) \arrow[ddd] \\
			& \check{H}^{i}_0(X,F') \arrow[r] \arrow[ul] \arrow[d] & H^1\left(k,\check{H}^{i-1}(\cl{X},\cl{F'})\right) \arrow[d] \arrow[ur]  \\
			& H^i_0(X,F') \arrow[r] \arrow[dl] & H^1\left(k,H^{i-1}(\cl{X},\cl{F'})\right) \arrow[dr] \\
			H_0^{i+1}(X,F) \arrow[rrr] &&& H^1\left(k,H^i(\cl{X},\cl{F})\right),
		\end{tikzcd}
		\]
		o\`u les fl\`eches obliques sont induites par les morphismes de bord dans le diagramme (\ref{derived-cech2}) associ\'e aux suites (\ref{eff'}) et (\ref{eff'cl}) et les fl\`eches horizontales et verticales viennent du diagramme (\ref{hochschild-cech}) pour $(F,i)$ et $(F',i-1)$. Les carr\'es de gauche et de droite commutent d'apr\`es le diagramme (\ref{derived-cech2}), la commutativit\'e du carr\'e du  haut suit d'un calcul explicite avec les cocycles et le carr\'e du bas anticommute par le lemme \ref{signe}. Comme $I$ et $\cl{I}$ sont flasques, les fl\`eches obliques de gauche sont des isomorphismes pour tout $i\geq 1$. Par r\'ecurrence sur $i$, on se r\'eduit alors \`a d\'emontrer la commutativit\'e du carr\'e (\ref{hochschild-cech}) dans le cas $i=0$.
		
		Supposons donc $i=0$. Dans ce cas, $\theta_X$ est un isomorphisme et son inverse est  un homomorphisme  de coin dans la suite spectrale (\ref{hochschild-serre-ss}). Comme $\theta_X$ est la limite inductive des $\theta_{X,k'}$ et $\check{\theta}_X$ est la limite inductive des $\check{\theta}_{X,k'}$, il suffit de prouver, pour toute extension galoisienne finie $k'/k$, la commutativit\'e du carr\'e
		\[
		\begin{tikzcd}
			\on{Ker}\left(\check{H}^1(X,F)\to \check{H}^1(X_{k'},F)\right) \arrow[d, "\text{$\wr$ $\phi$}"] \arrow[r,"\check{\theta}_{X,k'}"] & H^1\left(\on{Gal}(k'/k),F(X_{k'})\right) \arrow[d,equal] \\	
			\on{Ker}\left(H^1(X,F)\to H^1(X_{k'},F|_{X_{k'}})\right)\arrow[r,"\theta_{X,k'}"] & H^1\left(\on{Gal}(k'/k),F(X_{k'})\right),
		\end{tikzcd}
		\]
		o\`u $\phi$ est induit par l'isomorphisme (\ref{derived-cech}).
		Comme la suite spectrale (\ref{hochschild-serre-fini}) est isomorphe \`a la suite spectrale (\ref{cartan-leray}), on a un carr\'e commutatif
		\[
		\begin{tikzcd}
			\on{Ker}\left({H}^{1}(X,F)\to {H}^1(X_{k'},F|_{X_{k'}})\right) \arrow[d,equal] \arrow[r,"\theta_{X,k'}"] & H^1\left(\on{Gal}(k'/k),F(X_{k'})\right)	\arrow[d, "\text{$\wr$ $\rho$}"]  \\
			\on{Ker}\left(H^1(X,F)\to \check{H}^0\left(X_{k'}/X,\mc{H}^1(X,F)\right)\right)\arrow[r, "\eta_{X,k'}"] & \check{H}^1(X_{k'}/X,F),
		\end{tikzcd}
		\]
		o\`u l'isomorphisme $\eta_{X,k'}$ vient de la suite spectrale (\ref{cartan-leray}) et $\rho$ provient de l'isomorphisme (\ref{cech-galois-isom}). On se r\'eduit donc \`a prouver la commutativit\'e du carr\'e
		\[
		\begin{tikzcd}
			\on{Ker}\left(\check{H}^1(X,F)\to \check{H}^1(X_{k'},F)\right) \arrow[d,"\text{$\wr$ $\phi$}"] \arrow[r,"\check{\theta}_{X,k'}"] & H^1\left(\on{Gal}(k'/k),F(X_{k'})\right)	\arrow[d, "\text{$\wr$ $\rho$}"] \\
			\on{Ker}\left(H^1(X,F)\to \check{H}^0\left(X_{k'}/X,\mc{H}^1(X,F)\right)\right)\arrow[r, "\eta_{X,k'}"] & \check{H}^1(X_{k'}/X,F).	
		\end{tikzcd}
		\]
		L'inverse de $\eta_{X,k'}$ est un homomorphisme de coin dans la suite spectrale (\ref{cartan-leray}) et $\phi$ est d\'efini comme la limite inductive des homomorphismes de coin $\check{H}^1(U/X,F)\to H^1(X,F)$, o\`u $U\to X$ est un morphisme \'etale surjectif de type fini. Il s'en suit que l'inverse de $\eta_{X,k'}\circ \phi$ est l'homomorphisme canonique $\check{H}^1(X_{k'}/X,F)\to \check{H}^1(X,F)$, c'est-\`a-dire l'homomorphisme induit par l'identit\'e au niveau des cocycles de \v{C}ech. Pour conclure, il suffit alors de montrer que le triangle
		\begin{equation}\label{cech-derived-finaltr}
			\begin{tikzcd}
				\check{H}^1(X_{k'}/X,F)  \arrow[d,"\wr"] \arrow[r,"\rho"]  &H^1\left(\on{Gal}(k'/k),F(X_{k'})\right)\\ 
				\on{Ker}\left(\check{H}^1(X,F)\to \check{H}^1(X_{k'},F|_{X_{k'}})\right) \arrow[ur,swap, "\check{\theta}_{X,k'}"] 	
			\end{tikzcd}
		\end{equation}
		commute. 
		
		L'isomorphisme $\rho$ provient de l'identification (\ref{cech-galois-isom}) du complexe de \v{C}ech pour $F$ le long de $X_{k'}\to X$ avec le complexe des cocha\^{i}nes inhomog\`enes pour le $\on{Gal}(k'/k)$-module $F(X_{k'})$; voir \cite[Example III.2.6]{milne1980etale}. Plus pr\'ecis\'ement, soient $\alpha\in \check{H}^1(X_{k'}/X,F)$ et $\beta\in F(X_{k'}\times_X X_{k'})$ un cocycle de \v{C}ech repr\'esentant $\alpha$. Notons $\Gamma:=\on{Gal}(k'/k)$. L'isomorphisme \[X_{k'}\times \Gamma\stackrel{\sim}{\longrightarrow} X_{k'}\times_XX_{k'},\quad (s,g)\mapsto (s,sg)\] induit un isomorphisme
		\[F(X_{k'}\times_XX_{k'})\stackrel{\sim}{\longrightarrow} F(X_{k'}\times \Gamma)\simeq F(X_{k'})\times \Gamma\] 
		qui envoie $\beta$ vers $((\on{id},g)^*(\beta))_{g\in \Gamma}$. Ici $\on{id}=\on{id}_{X_{k'}}$ est l'identit\'e de $X_{k'}$ et $g:X_{k'}\to X_{k'}$ est l'isomorphisme induit par $g\in \Gamma$. Alors $\set{(\on{id},g)^*(\beta)}_{g\in \Gamma}$ est un $\Gamma$-cocycle et \[\rho(\alpha)=\left[\set{(\on{id},g)^*(\beta)}_{g\in \Gamma}\right]\in H^1\left(\Gamma,F(X_{k'})\right).\] 
		La fl\`eche verticale dans le triangle (\ref{cech-derived-finaltr}) envoie $\alpha$ sur $[\beta]\in \check{H}^1(X,F)$. Comme $d(\beta_{k'})=0$, il existe $\gamma \in F(X_{k'})$ tel que l'on ait $d(\gamma)=\beta_{k'}$. Par d\'efinition, $\check{\theta}_{X,k'}([\beta])=\left[\set{[g(\gamma)-\gamma]}_{g\in G}\right]$. Le diagramme commutatif de \cite[Example III.2.6]{milne1980etale} contient le carr\'e commutatif suivant:
		\[
		\begin{tikzcd}[column sep=huge, row sep=large]
			F(X_{k'}) \arrow[d,equal] \arrow[r, "p_0^*-p_1^*"] & F(X_{k'}\times_XX_{k'}) \arrow[d, "\text{$((\on{id},g)^*)_{g\in \Gamma}$}"]  \\
			F(X_{k'})   \arrow[r, "\left(g^*-\on{id}^*\right)_{g\in \Gamma}"] & F(X_{k'})\times \Gamma.   	
		\end{tikzcd}
		\]
		On en d\'eduit que l'on a $g(\gamma)-\gamma=(\on{id},g)^*(\beta)$ pour tout $g\in \Gamma$, donc $\theta_{X,k'}([\beta])=\rho(\alpha)$. Ceci prouve la commutativit\'e du triangle (\ref{cech-derived-finaltr}) et ach\`eve la d\'emonstration.
	\end{proof}
	
	\subsection{Compatibilit\'es}
	Les deux lemmes suivants montrent que $\theta_X$ est compatible avec les cup-produits et l'application trace en cohomologie \'etale. 
	
	\begin{lemma}\label{hochschild-cup}
		Soient $k$ un corps, $X$ une $k$-vari\'et\'e quasi-projective, $F$ et $F'$ deux faisceaux ab\'eliens pour la topologie \'etale sur $X$, puis $\alpha \in H^q(X,F')$, et $\cl{\alpha}$ l'image de $\alpha$ dans $H^q(\cl{X},\cl{F'})$.
		Alors pour tout $i\geq 0$ le diagramme
		\[
		\begin{tikzcd}
			\on{Ker}\left(H^{i+1}(X,F)\to H^{i+1}(\cl{X},\cl{F})\right) \arrow[r,"\theta_X"]  \arrow[d,"(-)\cup \alpha"]  & H^1\left(k,H^{i}(\cl{X},\cl{F})\right)  \arrow[d,"(-)\cup \cl{\alpha}"] &  \\
			\on{Ker}\left(H^{i+q+1}(X,F\otimes_{\Z} F')\to H^{i+q+1}(\cl{X},\cl{F}\otimes_{\Z} \cl{F'})\right) \arrow[r, "\theta_X"] & H^1\left(k,H^{i+q}(\cl{X},\cl{F}\otimes_{\Z} \cl{F'})\right)
		\end{tikzcd}
		\]
		$(-1)^q$-commute.
	\end{lemma}	
	On rappelle que, d'apr\`es \cite[Expos\'e IV, Proposition 13.4 (c)]{sga4I}, on a un isomorphisme canonique $\cl{F\otimes_{\Z}F'}\simeq \cl{F}\otimes_{\Z}\cl{F'}$.
	\begin{proof}
		L'isomorphisme (\ref{derived-cech}) respecte le cup-produit par \cite[Remark V.1.19(a)]{milne1980etale}. On note encore $\alpha\in \check{H}^q(X,F')$ la classe de \v{C}ech correspondant \`a $\alpha$. D'apr\`es le lemme
		\ref{cech-derived}, il suffit alors de v\'erifier l'\'egalit\'e $\check{\theta}_X(\beta)\cup \cl{\alpha}=\check{\theta}_X(\beta\cup \alpha)$ pour tout $\beta \in \on{Ker}(\check{H}^{i+1}(X,F)\to \check{H}^{i+1}(\cl{X},\cl{F}))$.
		
		Soient $U\to X$ un morphisme \'etale surjectif de type fini tel que $\alpha$ soit repr\'esent\'e par un cocycle $\alpha_0\in F'(U_X^{q+1})$ et $\cl{\alpha}_0$ l'image de $\alpha_0$ dans $\cl{F'}(\cl{U}^{q+1}_{\cl{X}})$. Soient $\beta_0\in F(U_X^{i+1})$ un cocycle qui repr\'esente $\beta$,  $\cl{\beta}_0$ l'image de $\beta_0$ dans $\cl{F}(\cl{U}_{\cl{X}}^{i+1})$ et $\gamma\in \cl{F}(\cl{U}^{i+1}_{\cl{X}})$ tel que l'on ait $d(\gamma)=\cl{\beta}_0$. Notons \[\on{pr}_1:U^{i+q+1}_X\to U^{i+1}_X,\qquad \on{pr}_2:U^{i+q+1}_X\to U^{q+1}_X\] les projections sur les $i+1$ premiers et $q+1$ derniers  facteurs, respectivement. Comme $\check{\theta}_X(\beta)$ est repr\'esent\'e par $\set{[g(\gamma)-\gamma]}_{g\in G}$, on voit que $\check{\theta}_X(\beta)\cup \cl{\alpha}$ est repr\'esent\'e par \begin{equation}\label{cocycle1}\big\{[(g(\on{pr}_1^*(\gamma))-\on{pr}_1^*(\gamma))\otimes \on{pr}_2^*(\cl{\alpha}_0)]\big\}_{g\in G}.
		\end{equation}
		
		D'un autre c\^{o}t\'e,	comme $d(\cl{\alpha}_0)=0$, la classe $\check{\theta}_X(\beta\cup \alpha)$ est repr\'esent\'ee par 
		\begin{equation}\label{cocycle2}
			\big\{[g(\on{pr}_1^*(\gamma)\otimes \on{pr}_2^*(\cl{\alpha}_0))-\on{pr}_1^*(\gamma)\otimes \on{pr}_2^*(\cl{\alpha}_0)]\big\}_{g\in G}.
		\end{equation}
		Comme $\cl{\alpha}_0$ est dans l'image de $F'(U^{q+1}_X) \to \cl{F'}(\cl{U}^{q+1}_{\cl{X}})$, on a $g(\cl{\alpha}_0)=\cl{\alpha}_0$ pour tout $g\in G$. On conclut que les cocycles (\ref{cocycle1}) et (\ref{cocycle2}) co\"{i}ncident et donc que l'on a $\check{\theta}_X(\beta)\cup \cl{\alpha}=\check{\theta}_X(\beta\cup \alpha)$, comme voulu.
	\end{proof}

	\begin{lemma}\label{hochschild-cores}
		Soient $k$ un corps, $f:X'\to X$ un morphisme fini \'etale de $k$-vari\'et\'es lisses et quasi-projectives, $\ell$ un nombre premier inversible dans $k$, $i\geq 0$ et $j$ des entiers. 
		\begin{enumerate}[label=(\alph*)]
		\item Le carr\'e
		\[
		\begin{tikzcd}
			\on{Ker}\left({H}^{i+1}(X',\Z_{\ell}(j))\to {H}^{i+1}(\cl{X'},\Z_{\ell}(j))\right) \arrow[r,"\theta_{X'}"] \arrow[d,"\on{tr}_f"] & H^1\left(k,{H}^i(\cl{X'},\Z_{\ell}(j))\right) \arrow[d,"H^1(\on{tr}_{\cl{f}})"]  \\	
			\on{Ker}\left({H}^{i+1}(X,\Z_{\ell}(j))\to {H}^{i+1}(\cl{X},\Z_{\ell}(j))\right) \arrow[r,"{\theta}_X"] & H^1\left(k,{H}^i(\cl{X},\Z_{\ell}(j))\right).
		\end{tikzcd}
		\]
		commute. 
		\item Soient $k\subset k'\subset \cl{k}$ une extension finie et s\'eparable et $H\subset G$ le groupe de Galois absolu de $k'$. Supposons que $X'=X_{k'}$ et que $f$ est la projection naturelle. Alors l'isomorphisme $G$-\'equivariant canonique \[H^i(\cl{X'},\Z_{\ell}(j))\simeq \Z[G/H]\otimes_{\Z} {H}^i(\cl{X},\Z_{\ell}(j))\]
		identifie la fl\`eche verticale de droite
		dans (a) avec la fl\`eche induite par la norme
		\[\Z[G/H]\otimes_{\Z} {H}^i(\cl{X},\Z_{\ell}(j))\longrightarrow {H}^i(\cl{X},\Z_{\ell}(j)).\] 
		\end{enumerate}
	\end{lemma}
	
	\begin{proof}
		(a) Soit $n\geq 1$ un entier. L'homomorphisme $\theta_X$ est compatible avec tout homomorphisme de faisceaux ab\'eliens \'etales sur $X$, et donc en particulier avec $\on{tr}_f:f^*f_*\mu_{{\ell}^n}\to \mu_{{\ell}^n}$. On a donc un carr\'e commutatif
		\begin{equation}\label{derived-fini}
			\begin{tikzcd}
				\on{Ker}\left({H}^{i+1}\left(X',\mu_{\ell^n}^{\otimes j}\right)\to {H}^{i+1}\left(\cl{X'},\mu_{\ell^n}^{\otimes j}\right)\right) \arrow[r,"{\theta}_{X'}"] \arrow[d,"\on{tr}_f"] & H^1\left(k,{H}^i\left(\cl{X'},\mu_{\ell^n}^{\otimes j}\right)\right) \arrow[d,"{H}^1(\on{tr}_{\cl{f}})"]  \\	
				\on{Ker}\left({H}^{i+1}\left(X,\mu_{\ell^n}^{\otimes j}\right)\to {H}^{i+1}\left(\cl{X},\mu_{\ell^n}^{\otimes j}\right)\right) \arrow[r,"{\theta}_X"] & H^1\left(k,{H}^i\left(\cl{X},\mu_{\ell^n}^{\otimes j}\right)\right).
			\end{tikzcd}
		\end{equation}
		Comme $\theta_X$, $\on{tr}_f$ et $\on{tr}_{\cl{f}}$ sont fonctoriels par rapport aux homomorphismes $\mu_{{\ell}^{n+1}}^{\otimes j}\to\mu_{{\ell}^n}^{\otimes j}$ de puissance $\ell$-i\`eme, les carr\'es (\ref{derived-fini}) avec $n\geq 1$ variable forment un syst\`eme inverse de diagrammes commutatifs compatibles. 
		
		On veut passer \`a la limite projective sur $n\geq 1$ dans le carr\'e (\ref{derived-fini}). Le foncteur de limite projective \'etant exact \`a gauche, on a une identification canonique
		\[\varprojlim_n \on{Ker}\left({H}^{i+1}\left(X,\mu_{\ell^n}^{\otimes j}\right)\to {H}^{i+1}\left(\cl{X},\mu_{\ell^n}^{\otimes j}\right)\right)\simeq \on{Ker}\left({H}^{i+1}(X,\Z_{\ell}(j))\to {H}^{i+1}(\cl{X},\Z_{\ell}(j))\right)\] et de m\^{e}me en rempla\c{c}ant $X$ par $X'$. Les groupes ${H}^i(\cl{X},\mu_{\ell^n}^{\otimes j})$ sont finis pour tout $i\geq 1$. D'apr\`es \cite[Corollary (2.7.6)]{neukirch2008cohomology} on a alors \[\varprojlim_n H^1\left(k,{H}^i\left(\cl{X},\mu_{\ell^n}^{\otimes j}\right)\right)\simeq H^1\left(k,{H}^i(\cl{X},\Z_{\ell}(j))\right)\] et de m\^{e}me pour $X'$.
		Le carr\'e commutatif de (a) est obtenu par passage \`a la limite sur $n\geq 1$ dans le carr\'e (\ref{derived-fini}).

		(b) L'identification canonique suit de la d\'ecomposition $G$-\'equivariante $\cl{X'}=\amalg_{g\in R} \cl{X}^g$, o\`u $R\subset G$ est un ensemble de repr\'esentants modulo $H$. Pour conclure, il suffit d'observer que, comme le morphisme $\cl{f}:\cl{X'}\to \cl{X}$ obtenu par changement de base de $f$ est un recouvrement \'etale trivial, pour tout faisceau \'etale $F$ sur $X$ l'application induite $\cl{f}_*\cl{f}^*\cl{F}\simeq \cl{F}^{[G:H]}\to \cl{F}$ est la somme.
	\end{proof}
	
	\subsection{Le cas d'une courbe}
	Soient $k$ un corps, $G$ le groupe de Galois absolu de $k$, $C$ une courbe projective, lisse et g\'eom\'etriquement connexe sur $k$ et $J_C$ la jacobienne de $C$. La suite exacte courte $G$-\'equivariante
	\begin{equation}\label{jacobian}
	0\longrightarrow J_C(\cl{k})\longrightarrow \on{Pic}(\cl{C})\xrightarrow{\on{deg}} \Z\longrightarrow 0
	\end{equation}
	identifie $J_C(\cl{k})$ au groupe des diviseurs de degr\'e $0$ sur $\cl{C}$, modulo \'equivalence rationnelle.
	Par passage aux $G$-invariants dans la suite (\ref{jacobian}), on obtient un diagramme commutatif
	\[
	\begin{tikzcd}
		& & \on{Pic}(C) \arrow[r, "\deg"] \arrow[d, hook]  & \Z \arrow[d, equal]  \\
		0 \arrow[r]  & J_C(k) \arrow[r]  & \on{Pic}(\cl{C})^G \arrow[r, "\deg"] & \Z
	\end{tikzcd}
	\]
	o\`u la suite du bas est exacte, et donc un homomorphisme injectif
	\begin{equation}\label{jac-deg0}
		\on{Ker}\left(\on{Pic}(C)\xrightarrow{\deg}\Z\right)\lhook\joinrel\longrightarrow J_C(k).
	\end{equation}
	Soit $n$ un entier inversible dans $k$. On a un diagramme commutatif
	\begin{equation}\label{pic-modn}
		\begin{tikzcd}
			\on{Pic}(C) \arrow[r] \arrow[d, "\deg"]  & \on{Pic}(C)/n \arrow[r,equal] \arrow[d, "\on{deg}"] & \arrow[d] \arrow[r,"\on{cl}_C"]  \on{Pic}(C)/n & H^2(C,\mu_n) \arrow[d]   \\
			\Z \arrow[r] &\Z/n &  \arrow[l,swap, "\sim"] \on{Pic}(\cl{C})/n \arrow[r,"\on{cl}_{\cl{C}}"] & H^2(\cl{C},\mu_n)    
		\end{tikzcd}
	\end{equation}
	qui induit un homomorphisme \[\on{cl}_C:\on{Ker}\left(\on{Pic}(C)\xrightarrow{\deg}\Z\right)\longrightarrow \on{Ker}\left(H^2(C,\mu_n)\to H^2(\cl{C},\mu_n)\right).\]
	La suite de Kummer $1\longrightarrow \mu_n\longrightarrow \G_{\on{m}}\longrightarrow \G_{\on{m}}\longrightarrow 1$ donne un isomorphisme $H^1(\cl{C},\mu_n)\stackrel{\sim}{\longrightarrow} H^1(\cl{C},\G_{\on{m}})[n]$. La suite obtenue par passage \`a la $n$-torsion dans la suite (\ref{jacobian}) est encore exacte, et donne alors un isomorphisme $J_C(\cl{k})[n]\simeq H^1(\cl{C},\G_{\on{m}})[n]$. On obtient un isomorphisme
	\begin{equation}\label{identification-taut} H^1(\cl{C},\mu_n)\stackrel{\sim}{\longrightarrow} H^1\left(\cl{C},\G_{\on{m}}\right)[n]\stackrel{\sim}{\longrightarrow} J_C(\cl{k})[n].
	\end{equation}
	
	Consid\'erons le diagramme:
	\begin{equation}\label{kummer-hochschild}
		\begin{tikzcd}
			\on{Ker}\left(\on{Pic}(C)\xrightarrow{\deg}\Z\right)/n  \arrow[d] \arrow[rr, "\on{cl}_C"] &  & \on{Ker}\left(H^2(C,\mu_n)\to H^2(\cl{C},\mu_n)\right)\arrow[d, "\theta_C"]    	\\			
			J_C(k)/n \arrow[r]  & H^1\left(k,J_C(\cl{k})[n]\right) & \arrow[l,swap, "\sim"] H^1\left(k,H^1(\cl{C},\mu_n)\right). 
		\end{tikzcd}	
	\end{equation}
	Ici, l'homomorphisme vertical de gauche est induit par la fl\`eche (\ref{jac-deg0}). Dans la ligne du bas de (\ref{kummer-hochschild}), la fl\`eche de gauche est induite par la suite exacte courte \[1\longrightarrow J_C[n]\longrightarrow J_C\xrightarrow{\times n} J_C\longrightarrow 1,\] 
	et la fl\`eche de droite est obtenue par passage \`a la cohomologie galoisienne dans l'isomorphisme (\ref{identification-taut}).

	\begin{lemma}\label{hochschild-courbe}
		Soit $C$ une courbe projective, lisse et g\'eom\'etriquement connexe sur le corps $k$ et soit $J_C$ la jacobienne de $C$. Pour tout entier $n\geq 1$ inversible dans $k$, le diagramme \emph{(\ref{kummer-hochschild})} commute.	
	\end{lemma}
	
	\begin{proof}
		Soit $D\in \on{Pic}(C)$ tel que $\on{deg}(D)=0$. Dans le diagramme (\ref{kummer-hochschild}), la fl\`eche compos\'ee  \[\on{Ker}\left(\on{Pic}(C)\xrightarrow{\deg}\Z\right)/n\to J_C(k)/n \to  H^1\left(k,J_C(\cl{k})[n]\right)\] envoie $D$ sur
		la classe repr\'esent\'ee par le cocycle $\set{[g(\tilde{D})-\tilde{D}]}_{g\in G}$, o\`u $\tilde{D}\in J_C(\cl{k})$ satisfait $n\tilde{D}=D$. 		
		
		D'apr\`es le lemme \ref{cech-derived} pour $i=1$, il suffit de prouver que le diagramme obtenu de (\ref{kummer-hochschild}) en rempla\c{c}ant la cohomologie des faisceaux par la cohomologie de \v{C}ech anticommute. Soient $U\to C$ un morphisme \'etale surjectif de type fini, et $\lambda\in \G_{\on{m}}(U^2_C)$ un cocycle repr\'esentant $D$ dans $\check{H}^1(C,\G_{\on{m}})$. Pour tout $h\in \set{0,1}$ et $i,j\in \set{0,1,2}$, les projections
		\[p_h:U^2_C\to U\quad\text{et}\quad p_{ij}: U^3_C\to U^2_C\]
		induisent des diagrammes commutatifs
		\[
		\begin{tikzcd}
			\mu_n(U) \arrow[r, "p_h^*"] \arrow[d, hook] & \mu_n(U^2_C) \arrow[d, hook] \arrow[r, "p_{ij}^*"]  & \mu_n(U^3_C) \arrow[d, hook] \\
			\G_{\on{m}}(U) 	\arrow[r, "p_h^*"] & \G_{\on{m}}(U^2_C) \arrow[r, "p_{ij}^*"]  & \G_{\on{m}}(U^3_C). 
		\end{tikzcd}
		\]
		On note $\cl{p}_h^{\, *}$ et $\cl{p}_{ij}^{\, *}$ les fl\`eches correspondantes au niveau $\cl{k}$.
		
		Le fait que $\lambda$ est un cocycle s'\'ecrit
		\begin{equation*}\label{cocycle-lambda}p_{12}^*(\lambda)p_{02}^*(\lambda)^{-1}p_{01}^*(\lambda)=1 \text{ dans $\G_{\on{m}}(U_C^3)$.}
		\end{equation*}
		Quitte \`a raffiner $U$, on peut supposer qu'il existe $\lambda_1\in \G_{\on{m}}(U^2_C)$ tel que $\lambda=\lambda_1^n$. On appelle $\cl{\lambda}$ et $\cl{\lambda}_1$  les images de $\lambda$ et $\lambda_1$ dans $\G_{\on{m}}(\cl{U}^2_{\cl{C}})$, respectivement. (On note que $\lambda_1$ et $\cl{\lambda}_1$ ne respectent pas n\'ecessairement la condition de cocycle.) Soit encore $\tilde{\lambda}\in \G_{\on{m}}(\cl{U}^2_{\cl{C}})$ un cocycle repr\'esentant $\tilde{D}$ dans $\check{H}^1(\cl{C},\G_{\on{m}})$. Comme $D=n\tilde{D}$, $[\cl{\lambda}]=[\tilde{\lambda}^n]$ dans $\check{H}^1(\cl{C},\G_{\on{m}})$, et donc il existe $\xi\in \G_{\on{m}}(\cl{U})$ tel que l'on ait
		\[\cl{\lambda}=\cl{p}_1^{\, *}(\xi)\cl{p}_0^{\, *}(\xi)^{-1}\tilde{\lambda}^n \text{ dans $\G_{\on{m}}(\cl{U}^2_{\cl{C}})$}.\] 
		Quitte \`a raffiner $U$, on peut supposer qu'il existe $\xi_1\in \G_{\on{m}}(\cl{U})$ tel que $\xi=\xi_1^n$. On d\'efinit
		\begin{equation}\label{mu-n}
			\nu:=\cl{\lambda}_1(\cl{p}_1^{\, *}(\xi_1)\cl{p}_0^{\, *}(\xi_1)^{-1}\tilde{\lambda})^{-1} \in  \mu_n(\cl{U}^2_{\cl{C}}).
		\end{equation}
		
		Par la d\'efinition des homomorphismes de bord en cohomologie de \v{C}ech, $\on{cl}_C([\lambda])$ est repr\'esent\'e par le cocycle \[p_{12}^*(\lambda_1)p_{02}^*(\lambda_1)^{-1}p_{01}^*(\lambda_1)\in \mu_n(U^3_C).\] L'image de ce cocycle dans $\mu_n(\cl{U}^3_{\cl{C}})$ est le cocycle
		\[\cl{p}_{12}^{\, *}(\cl{\lambda}_1)\cl{p}_{02}^{\, *}(\cl{\lambda}_1)^{-1}\cl{p}_{01}^{\, *}(\cl{\lambda}_1)\in \mu_n(\cl{U}^3_{\cl{C}}).\]
		On a
		\[\cl{p}_{12}^{\, *}\left(\cl{p}_1^{\, *}(\xi_1)\cl{p}_0^{\, *}(\xi_1)^{-1}\right)\cl{p}_{02}^{\, *}\left(\cl{p}_1^{\, *}(\xi_1)\cl{p}_0^{\, *}(\xi_1)^{-1}\right)^{-1}\cl{p}_{01}^{\, *}\left(\cl{p}_1^{\, *}(\xi_1)\cl{p}_0^{\, *}(\xi_1)^{-1}\right)=1\]
		et on sait que $\cl{p}_{12}^{\, *}(\tilde{\lambda})\cl{p}_{02}^{\, *}(\tilde{\lambda})^{-1}\cl{p}_{01}^{\, *}(\tilde{\lambda})=1$ parce que $\tilde{\lambda}$ est un cocycle. On d\'eduit de la d\'efinition (\ref{mu-n}) que l'on a
		\begin{equation}\label{mu-n'}
			\cl{p}_{12}^{\, *}(\cl{\lambda}_1)\cl{p}_{02}^{\, *}(\cl{\lambda}_1)^{-1}\cl{p}_{01}^{\, *}(\cl{\lambda}_1)=\cl{p}_{12}^{\, *}(\nu)\cl{p}_{02}^{\, *}(\nu)^{-1}\cl{p}_{01}^{\, *}(\nu) \text{ dans}\ \mu_n(\cl{U}^3_{\cl{C}}).
		\end{equation} 
		Comme $\nu \in \mu_n(\cl{U}^2_{\cl{C}})$, la classe $\check{\theta}_C(\on{cl}_C[\lambda])\in H^1(k,\check{H}^1(\cl{C},\mu_n))$ est repr\'esent\'ee par le cocycle $\set{[g(\nu)/\nu]}_{g\in G}$.
		
		L'isomorphisme $\check{H}^1(\cl{C},\mu_n)\xrightarrow{\sim} \check{H}^1(\cl{C},\G_{\on{m}})[n]$ est induit par l'inclusion $\mu_n\subset \G_{\on{m}}$: si $c\in \mu_n(\cl{U}^2_{\cl{C}})$, l'image de $[c]\in \check{H}^1(\cl{C},\mu_n)$ est la classe de $c$ dans $\check{H}^1(\cl{C},\G_{\on{m}})$, o\`u on voit $c$ comme un \'el\'ement de $\G_{\on{m}}(\cl{U}^2_{\cl{C}})$. Donc l'image de $\check{\theta}_C(\on{cl}_C[\lambda])$ dans $H^1(k,\check{H}^1(\cl{C},\G_{\on{m}})[n])$ est $\big\{[g(\nu)/\nu]\big\}_{g\in G}$, o\`u on voit chaque $g(\nu)/\nu$ comme un \'el\'ement de $\G_{\on{m}}(\cl{U}^2_{\cl{C}})$.
		
		Comme $\cl{\lambda}_1\in \G_{\on{m}}(\cl{U}^2_{\cl{C}})$ est l'image de $\lambda_1\in \G_{\on{m}}(U^2_C)$, on a $g(\cl{\lambda}_1)=\cl{\lambda}_1$ pour tout $g\in G$. D'apr\`es la d\'efinition (\ref{mu-n}), pour tout $g\in G$ on a alors l'\'egalit\'e
		\[\frac{g(\nu)}{\nu}= \frac{g\left(\cl{p}_1^{\, *}(\xi_1)^{-1}\cl{p}_0^{\, *}(\xi_1)\right)}{\cl{p}_1^{\, *}(\xi_1)^{-1}\cl{p}_0^{\, *}(\xi_1)}\cdot \frac{g(\tilde{\lambda}^{-1})}{\tilde{\lambda}^{-1}} \text{ dans}\ \G_{\on{m}}(\cl{U}^2_{\cl{C}}),\]
		donc
		\[
		[g(\nu)/\nu]=[g(\tilde{\lambda}^{-1})/\tilde{\lambda}^{-1}] \text{ dans $\check{H}^1(\cl{C},\G_{\on{m}})[n]$}.
		\]
		
		On conclut que l'image de $\check{\theta}_C(\on{cl}_C([\lambda]))$ dans $H^1(k, \check{H}^1(\cl{C},\G_{\on{m}})[n])$ est repr\'esent\'ee par $\set{[g(\tilde{\lambda}^{-1})/\tilde{\lambda}^{-1}]}_{g\in G}$. Comme $\tilde{D}$ est repr\'esent\'e par $\tilde{\lambda}$, l'image de $\check{\theta}_C(\on{cl}_C([\lambda]))$ dans $H^1(k, J(\cl{k})[n])$ est repr\'esent\'ee par $\set{[g(-\tilde{D})-(-\tilde{D})]}_{g\in G}=\set{-[(g(\tilde{D})-\tilde{D})]}_{g\in G}$. Ceci montre que le diagramme obtenu de (\ref{kummer-hochschild}) par passage \`a la cohomologie de \v{C}ech anticommute, donc que le diagramme (\ref{kummer-hochschild}) commute.
	\end{proof}

	%\section{Classes de cohomologie g\'eom\'etriquement triviales}
	\section{Preuve du th\'eor\`eme \ref{mainthm}}\label{sec4}
	
	Le but de cette section est la d\'emonstration du th\'eor\`eme suivant et, par cons\'equent, du th\'eor\`eme \ref{mainthm}. 
	
	\begin{thm}\label{iota-dans-image}
		Soient $\F$ un corps fini, $\ell$ un nombre premier inversible dans $\F$, $C$ et $S$ deux vari\'et\'es g\'eom\'etriquement connexes, projectives et lisses sur $\F$, de dimension $1$ et $2$ respectivement et $X:=C\times_{\F} S$. 
		Sous l'hypoth\`ese $b_{2}(\cl{S})=\rho(\cl{S})$, 
		le noyau de $H^{4}(X,\Z_{\ell}(2)) \to H^{4}(\cl{X},\Z_{\ell}(2))$	est contenu dans l'image de l'application cycle
		\[ \on{cl}_X: CH^2(X)\otimes_{\Z} \Z_{\ell} \longrightarrow H^{4}(X,\Z_{\ell}(2)).\]
	\end{thm}
	
	Soit $X$ une vari\'et\'e projective et lisse sur un corps fini $\F$. Comme $\on{cd}(\F)=1$, pour tout faisceau ab\'elien fini $F$ sur $X$ et pour tout $i\geq 0$ l'homomorphisme $\theta_X$ de (\ref{hochschild-derived}) est un isomorphisme. Soit $j$ un entier. Pour tout $n\geq 1$, on pose $F=\mu_{\ell^n}^{\otimes j}$ dans (\ref{hochschild-derived}). On obtient un syst\`eme inverse d'isomorphismes
	\[\theta_X:\on{Ker}\left(H^{i+1}\left(X,\mu_{{\ell}^n}^{\otimes j}\right)\to H^{i+1}\left(\cl{X},\mu_{{\ell}^n}^{\otimes j}\right)\right)\stackrel{\sim}{\longrightarrow} H^1\left(\F,H^i\left(\cl{X},\mu_{{\ell}^n}^{\otimes j}\right)\right).\]
	En utilisant le lemme \ref{limit-coho}(b), la finitude de la cohomologie \'etale \`a coefficients $\mu_{{\ell}^n}^{\otimes j}$ pour les $\F$-vari\'et\'es projectives lisses et l'exactitude \`a gauche du foncteur de limite projective, on obtient par passage \`a la limite projective un isomorphisme
	\begin{equation}\label{just-for-intro}\theta_X:\on{Ker}\left(H^{i+1}(X,\Z_{\ell}(j))\to H^{i+1}(\cl{X},\Z_{\ell}(j))\right)\stackrel{\sim}{\longrightarrow} H^1\left(\F,H^i(\cl{X},\Z_{\ell}(j))\right).
	\end{equation}
	Comme $\F$ est fini, $H^{i+1}(X,\Z_{\ell}(j))\longrightarrow H^{i+1}(\cl{X},\Z_{\ell}(j))^G$ est surjectif par le lemme \ref{limit-coho}(b). On d\'eduit de l'isomorphisme (\ref{just-for-intro}) l'existence d'une suite exacte courte
	\begin{equation}\label{hochschild}
	\begin{tikzcd}[column sep=normal]		
		0\ar[r] & H^1\left(\F,H^{i}(\cl{X},\Z_{\ell}(j))\right)\ar[r,"\iota_X"] & H^{i+1}(X,\Z_{\ell}(j))\ar[r] & H^{i+1}(\cl{X},\Z_{\ell}(j))^G \ar[r] & 0.
		\end{tikzcd}
	\end{equation}
	o\`u $\iota_X$ est induit par l'inverse de $\theta_X$. Comme $\theta_X$ est fonctoriel en $X$, la suite (\ref{hochschild}) est naturelle en $X$.
	
	Soient $C$ et $S$ deux vari\'et\'es g\'eom\'etriquement connexes, projectives et lisses sur $\F$, de dimension $1$ et $2$ respectivement, et $X:=C\times_{\F} S$.
	Comme $H^*(\cl{C},\Z_{\ell})$ est sans torsion, par la formule de K\"unneth $\ell$-adique \cite[Corollary VI.8.13]{milne1980etale} on a des isomorphismes Galois-\'equivariants
	\begin{equation}\label{kunneth3}H^3\left(\cl{X},\Z_{\ell}(2)\right)\simeq H^3\left(\cl{S},\Z_{\ell}(2)\right)\oplus \left[ H^2\left(\cl{S},\Z_{\ell}(1)\right)\otimes_{\Z_{\ell}} H^1\left(\cl{C},\Z_{\ell}(1)\right) \right]\oplus H^1(\cl{S},\Z_{\ell}(1)),
	\end{equation}
	\begin{equation}\label{kunneth4}H^4\left(\cl{X},\Z_{\ell}(2)\right)\simeq H^4\left(\cl{S},\Z_{\ell}(2)\right)\oplus \left[ H^3\left(\cl{S},\Z_{\ell}(1)\right)\otimes_{\Z_{\ell}} H^1\left(\cl{C},\Z_{\ell}(1)\right) \right]\oplus H^2(\cl{S},\Z_{\ell}(1)).
	\end{equation}	
	induits par projections et cup-produits.
	Nous allons obtenir
	le th\'eor\`eme \ref{iota-dans-image} comme cons\'equence des trois lemmes suivants.
	
	\begin{lemma}\label{lemmeA}
		Soient $C$ et $S$ deux vari\'et\'es g\'eom\'etriquement connexes, projectives et lisses sur $\F$, de dimension $1$ et $2$ respectivement et $X:=C\times_{\F} S$.
		L'image de  \emph{(\ref{tate-int3})} contient le facteur direct $H^1\left(\F,H^3(\cl{S},\Z_{\ell}(2))\right)$ de  $H^1\left(\F,H^3(\cl{X},\Z_{\ell}(2))\right)$.	
	\end{lemma}
	
	\begin{proof}
		Par un th\'eor\`eme de Kato et Saito \cite{kato1983unramified} (voir aussi \cite[Th\'eor\`eme 5]{colliot1983torsion}), l'application cycle $CH^2(S)\otimes_\Z \Z_{\ell} \longrightarrow H^4(S,\Z_{\ell}(2))$ est un isomorphisme. Pour nos besoins, il nous suffira de savoir que cette application cycle est surjective, ce qui a \'et\'e d\'emontr\'e par Lang \cite{lang1956unramified}. On a un diagramme commutatif
		\[
		\begin{tikzcd}
			CH^2(S)\otimes_\Z \Z_{\ell} \arrow[r,"\on{cl}_S"] \arrow[d, "\on{pr}_S^*"]  & H^4(S,\Z_{\ell}(2)) \arrow[d, "\on{pr}_S^*"]   & H^1\left(\F, H^3(\cl{S},\Z_{\ell}(2))\right) \arrow[l,swap, "\iota_S"] \arrow[d, "\on{pr}_S^*"]  \\
			CH^2(X)\otimes_\Z \Z_{\ell} \arrow[r, "\on{cl}_X"]  & H^4(X,\Z_{\ell}(2))  & \arrow[l,swap,"\iota_X"] H^1(\F, H^3\left(\cl{X},\Z_{\ell}(2))\right), 
		\end{tikzcd}
		\]
		o\`u le carr\'e de gauche commute par \cite[Proposition VI.9.2]{milne1980etale} (dont la preuve vaut sur un corps de base quelconque) et le carr\'e de droite commute par la naturalit\'e de la suite exacte (\ref{hochschild}). Ceci implique que $H^1\left(\F,H^3(\cl{S},\Z_{\ell}(2))\right)$ est dans l'image de  (\ref{tate-int3}), comme voulu.	
	\end{proof}
	
	\begin{rmk}\label{rmkA}
		Soit $X=Y\times_{\F} C$, o\`u 
		$C$ est une courbe projective, lisse et g\'eom\'e\-tri\-quement connexe, $Y$ est une vari\'et\'e  projective, lisse et g\'eom\'e\-tri\-quement connexe de dimension $d$ telle que l'image de l'application (\ref{tate-int3}) pour $Y$ contient le sous-groupe \[H^1\left(\F,H^{2d-1}(\cl{Y},\Z_{\ell}(d))\right)\subset H^{2d}(Y,\Z_{\ell}(d)).\] L'argument de la preuve du lemme \ref{lemmeA} montre que le facteur direct de K\"unneth
		\[H^1\left(\F,H^{2d-1}(\cl{Y},\Z_{\ell}(d))\right)\subset H^1\left(\F,H^{2d+1}(\cl{X},\Z_{\ell}(d))\right)\]
		est dans l'image de l'application (\ref{tate-int3}) pour~$X$.
	\end{rmk}

	\begin{lemma}\label{lemmeB}
		Soient $C$ et $S$ deux vari\'et\'es g\'eom\'etriquement connexes, projectives et lisses sur $\F$, de dimension $1$ et $2$ respectivement et $X:=C\times_{\F} S$.
		L'image de l'application \emph{(\ref{tate-int3})} contient le facteur direct $H^1\left(\F, H^1(\cl{S},\Z_{\ell}(1))\right)$  de  $H^1\left(\F,H^3(\cl{X},\Z_{\ell}(2))\right)$.	
	\end{lemma}
	
	\begin{proof}
		La suite de Kummer induit, par passage \`a la limite projective, une suite exacte
		\begin{equation}\label{kummer-s}\on{Pic}(S)\otimes_{\Z}\Z_{\ell}\longrightarrow H^2(S,\Z_{\ell}(1))\longrightarrow T_{\ell}(\on{Br}(S)).
		\end{equation}
		Montrons que l'application compos\'ee
		\[
		\begin{tikzcd}
		H^1\left(\F,H^1(\cl{S},\Z_{\ell}(1))\right)\ar[r,"\iota_S"] & H^2(S,\Z_{\ell}(1))\ar[r] & T_{\ell}(\on{Br}(S))
\end{tikzcd}		
		\]
		est nulle.	Un argument de poids, utilisant les conjectures de Weil d\'emontr\'ees par Deligne \cite{weil1}, montre que $H^1\left(\F,H^1(\cl{S},\Z_{\ell}(1))\right)$ est un groupe fini; voir \cite[p. 781]{colliot1983torsion}. 	Comme tout module de Tate, le groupe $T_{\ell}(\on{Br}(S))$ est sans torsion.	L'exactitude de la suite (\ref{kummer-s}) entra\^{i}ne donc que $H^1\left(\F, H^1(\cl{S},\Z_{\ell}(1))\right)$ est dans l'image de \[\on{Pic}(S)\otimes_{\Z}\Z_{\ell}\longrightarrow H^2(S,\Z_{\ell}(1)).\]
		%Et plus pr\'ecis\'ement de $Pic(S)_{tors} \otimes \Z_{\ell}$, et encore plus pr\'ecis\'ement dans l'image du noyau de la fl\`eche compos\'ee $Pic(S)_{tors} \otimes \Z_{\ell}  \to  H^2(S,\Z_{\ell}(1)) \to H^2(\cl{S},\Z_{\ell}(1)$.
		
		D'apr\`es la borne de Hasse--Weil, il existe un $0$-cycle $\zeta\in \on{Pic}(C)$ de degr\'e $1$; voir par exemple \cite[1.5.3. Lemme 1]{soule1984groupes}. Soit $\cl{\on{cl}_C(\zeta)}$ l'image de $\on{cl}_C(\zeta)\in H^2(C,\mu_n)$ dans $H^2(\cl{C},\mu_n)$. On a un diagramme commutatif
		\[
		\begin{tikzcd}[row sep=large]
			\on{Pic}(S)\otimes_\Z \Z_{\ell} \arrow[r,"\on{cl}_S"] \arrow[d, "\on{pr}_S^*(-)\cup \on{pr}_C^*(\xi)"]  & H^2(S,\Z_{\ell}(1)) \arrow[d, "\on{pr}_{S}^*(-)\cup \on{pr}^*_{C}(\on{cl}_C(\zeta))"]   & H^1\left(\F, H^1(\cl{S},\Z_{\ell}(1))\right) \arrow[l,swap, "\iota_S"] \arrow[d, "\on{pr}_{\cl{S}}^*(-)\cup \on{pr}^*_{\cl{C}}(\cl{\on{cl}_C(\zeta)})"]  \\
			CH^2(X)\otimes_\Z \Z_{\ell} \arrow[r, "\on{cl}_X"]  & H^4(X,\Z_{\ell}(2))  & \arrow[l,swap,"\iota_X"] H^1\left(\F, H^3(\cl{X},\Z_{\ell}(2))\right).
		\end{tikzcd}
		\]
		Le carr\'e de gauche commute par \cite[Proposition VI.9.4]{milne1980etale} (dont la preuve vaut sur un corps de base quelconque) et le carr\'e de droite commute par les lemmes \ref{hochschild-cup} et \ref{limit-coho}(b). 	Ceci implique que $H^1\left(\F, H^1(\cl{S},\Z_{\ell}(1))\right)$ est dans l'image de l'application (\ref{tate-int3}).
	\end{proof}
	
	\begin{rmk}
	\leavevmode
	\begin{enumerate}[label=(\roman*)]
	\item Plus pr\'ecis\'ement, $H^1\left(\F, H^1(\cl{S},\Z_{\ell}(2))\right)$ est dans l'image du noyau de la fl\`eche compos\'ee \[\on{Pic}(S)_{\on{tors}} \otimes_{\Z} \Z_{\ell}  \longrightarrow  H^2(S,\Z_{\ell}(1)) \longrightarrow H^2(\cl{S},\Z_{\ell}(1)).\]
	\item Soit $X=Y\times_{\F} C$, o\`u $C$ est une courbe projective, lisse et g\'eom\'etriquement connexe, $Y$ est une vari\'et\'e  projective, lisse et g\'eom\'etriquement connexe de dimension $d$ telle que $H^1(\F,H^{2d-3}(\cl{Y},\Z_{\ell}(d-1)))$ est dans l'image de l'application (\ref{tate-int3}) pour $Y$. Alors l'argument de la preuve du lemme \ref{lemmeB} montre que le facteur de K\"unneth $H^1\left(\F,H^{2d-3}(\cl{Y},\Z_{\ell}(d-1))\right)\subset H^1\left(\F,H^{2d-1}(\cl{X},\Z_{\ell}(d))\right)$ est dans l'image de l'application (\ref{tate-int3}) pour~$X$.
	\end{enumerate}
	\end{rmk}

	\begin{lemma}\label{lemmeC}
		Soient $C$ et $S$ deux vari\'et\'es g\'eom\'etriquement connexes, projectives et lisses sur $\F$, de dimension $1$ et $2$ respectivement et $X:=C\times_{\F} S$.
		Supposons que l'on ait $b_2(\cl{S})=\rho(\cl{S})$ et que $G$ agisse trivialement sur $\on{NS}(\cl{S})$. Alors  l'image de l'application \emph{(\ref{tate-int3})} contient le facteur direct	$H^1\left(\F,H^2(\cl{S},\Z_{\ell}(1))\otimes_{\Z_{\ell}} H^1(\cl{C},\Z_{\ell}(1))\right)$ de $H^1\left(\F, H^3(\cl{X},\Z_{\ell}(2))\right)$.
	\end{lemma}

	\begin{proof}
		Soit $n$ un entier inversible dans $\F$. On pose 
		\begin{align*}
			\on{Pic}_0(C)&:=\on{Ker}\left(\on{Pic}(C)\xrightarrow{\deg}\Z\right), \\
			(CH^2(X)/n)_0&:=\on{Ker}\left(CH^2(X)/n\longrightarrow CH^2(\cl{X})/n\right), \\
			H^2_0(C,\mu_n)&:=\on{Ker}\left(H^2(C,\mu_n)\longrightarrow H^2(\cl{C},\mu_n)\right), \quad\text{et} \\
			H^4_0\left(X,\mu_n^{\otimes 2}\right)&:=\on{Ker}\left(H^2\left(X,\mu_n^{\otimes 2}\right)\longrightarrow H^2\left(\cl{X},\mu_n^{\otimes 2}\right)\right).
		\end{align*}
		Consid\'erons le diagramme commutatif
		\begin{equation}\label{gros}
			\adjustbox{scale=0.9,center}{ 
				\begin{tikzcd}
					\on{Pic}(S)\otimes_{\Z} J_C(\F)/n \arrow[d, "\wr"]  & \arrow[l, swap, "\sim"] \on{Pic}(S)\otimes_{\Z}\on{Pic}_0(C)/n \arrow[r, "\cup"] \arrow[d, "\on{cl}_S\otimes \on{cl}_C"] & (CH^2(X)/n)_0 \arrow[d, "\on{cl}_X"] \\
					\on{Pic}(S)\otimes_{\Z}H^1\left(\F, J_C(\cl{\F})[n]\right) &  H^2(S,\mu_n)\otimes_{\Z} H^2_0(C,\mu_n)\arrow[r,"\cup"] \arrow[d, "\cl{(-)}\otimes\theta_C"] & H^4_0\left(X,\mu_n^{\otimes 2}\right) \arrow[dd, "\theta_X"] \\
					\on{Pic}(S)\otimes_{\Z}H^1\left(\F, H^1\left(\F,H^1(\cl{C},\mu_n)\right)\right) \arrow[u, swap, "\wr"] \arrow[r] \arrow[d, "\cup"]   & H^2(\cl{S},\mu_n)^G\otimes_{\Z}H^1\left(\F, H^1(\cl{C},\mu_n)\right) \arrow[d, "\cup"]    \\
					H^1\left(\F, \on{Pic}(\cl{S})\otimes_{\Z}H^1\left(\F, H^1(\cl{C},\mu_n)\right)\right) \arrow[r] &H^1\left(\F, H^2(\cl{S},\mu_n)\otimes_{\Z} H^1(\cl{C},\mu_n)\right)\arrow[r] & H^1\left(\F, H^3\left(\cl{X},\mu_n^{\otimes 2}\right)\right).
				\end{tikzcd}
			}
		\end{equation}
		La fl\`eche horizontale sup\'erieure gauche est (\ref{jac-deg0}). Comme $\on{Br}(\F)=0$, les fl\`eches naturelles $\on{Pic}(C)\to \on{Pic}(\cl{C})^G$ et $\on{Pic}(S)\to \on{Pic}(\cl{S})^G$ sont bijectives. En particulier l'homomorphisme (\ref{jac-deg0}) est un isomorphisme pour $C$. Le rectangle de gauche est induit par (\ref{kummer-hochschild}) et commute d'apr\`es le lemme \ref{hochschild-courbe}, le carr\'e commutatif sup\'erieur droit vient de la compatibilit\'e de l'application cycle et du cup-produit. La commutativit\'e du carr\'e inf\'erieur gauche suit de la fonctorialit\'e du cup-produit en cohomologie galoisienne, et celle du carr\'e inf\'erieur droit se d\'eduit de la commutativit\'e, pour tout $\alpha\in H^2(S,\mu_n)$, du carr\'e
		\[
		\begin{tikzcd}[column sep=large, row sep=large]
			H^1\left(\F, H^1(\cl{C},\mu_n)\right) \arrow[r, "\iota_C"] \arrow[d, "H^1(\cl{\alpha}\cup\on{pr}^*_{\cl{C}}(-))"]  & H^2(C,\mu_n) \arrow[d, "\alpha\cup\on{pr}_{C}^*(-)"] \\
			H^1\left(\F, H^3(\cl{X},\mu_n)\right) \arrow[r, "\iota_C"] & H^4(X,\mu_n),   
		\end{tikzcd}
		\]
		ce qui suit du lemme \ref{hochschild-cup}.	
		
		On veut passer \`a la limite sur $n=\ell^m$, $m\geq 0$, dans le diagramme (\ref{gros}). Comme $\on{Pic}^0_{S/\F}(\cl{\F})$ est divisible, en tensorisant la suite exacte courte
		\begin{equation}\label{pic0-pic-h2}
		\begin{tikzcd}
			0\ar[r] & \on{Pic}^0_{S/\F}(\cl{\F})\ar[r] & \on{Pic}(\cl{S})\ar[r] & \on{NS}(\cl{S})\ar[r] & 0
		\end{tikzcd}
		\end{equation} 	
		par le groupe fini $H^1(\cl{C},\mu_{{\ell}^m})$, on obtient un isomorphisme
		\[\on{Pic}(\cl{S})\otimes_{\Z}H^1(\cl{C},\mu_{{\ell}^m})\simeq \on{NS}(\cl{S})\otimes_{\Z} H^1(\cl{C},\mu_{{\ell}^m}).\]
		Comme $\on{NS}(\cl{S})$ est un $\Z$-module de type fini, le lemme \ref{tensor-limit}(b) entra\^{\i}ne
		\[\varprojlim_m \left(\on{Pic}(\cl{S})\otimes_{\Z} H^1(\cl{C},\mu_{\ell^m})\right)\simeq \on{NS}(\cl{S})\otimes_{\Z}H^1(\cl{C},\Z_{\ell}).\]
		Les groupes $J_C(\F)$ et $J_C(\cl{\F})[n]$ sont finis, $\on{Pic}(S)=H^0(\F, \on{Pic}(\cl{S}))$ est un $\Z$-module de type fini, les groupes de cohomologie \'etale des vari\'et\'es $\cl{C},\cl{S}$ et $\cl{X}$ \`a valeurs dans $\mu_n^{\otimes j}$ sont finis, et ceux \`a valeurs dans $\Z_{\ell}$ sont des $\Z_{\ell}$-modules de type fini. Par passage \`a la limite dans le diagramme commutatif (\ref{gros}) sur $n=\ell^m$, $m\geq 0$, en utilisant les lemmes \ref{tensor-limit}, \ref{lim-tensor-lim} et \ref{limit-coho}(b), on obtient alors le diagramme commutatif suivant:
		\begin{equation}\label{gros'}
			\adjustbox{scale=0.9,center}{ 
				\begin{tikzcd}
					\on{Pic}(S)\otimes_{\Z} J_C(\F)\otimes_{\Z}\Z_{\ell} \arrow[d, "\wr"]  & \arrow[l, swap, "\sim"] \on{Pic}(S)\otimes_{\Z}\on{Pic}_0(C)\otimes_{\Z}\Z_{\ell} \arrow[r, "\cup"] \arrow[d, "\on{cl}_S\otimes \on{cl}_C"] & \varprojlim(CH^2(X)/\ell^m)_0 \arrow[d, "\on{cl}_X"] \\
					\on{Pic}(S)\otimes_{\Z}H^1\left(\F, T_{\ell}(J_C)\right) &  H^2(S,\Z_{\ell}(1))\otimes_{\Z_{\ell}} H^2_0(C,\Z_{\ell}(1))\arrow[r,"\cup"] \arrow[d, "\cl{(-)}\otimes\theta_C"] & H^4_0(X,\Z_{\ell}(2)) \arrow[dd, "\theta_X"] \\
					\on{Pic}(S)\otimes_{\Z}H^1\left(\F, H^1(\cl{C},\Z_{\ell}(1))\right) \arrow[u, swap, "\wr"] \arrow[r] \arrow[d, "\cup"]   & H^2(\cl{S},\Z_{\ell}(1))^G\otimes_{\Z_{\ell}}H^1\left(\F, H^1(\cl{C},\Z_{\ell}(1))\right) \arrow[d, "\cup"]    \\
					H^1\left(\F, \on{NS}(\cl{S})\otimes_{\Z}H^1(\cl{C},\Z_{\ell}(1))\right) \arrow[r] &H^1\left(\F, H^2(\cl{S},\Z_\ell(1))\otimes_{\Z_\ell} H^1(\cl{C},\Z_{\ell}(1))\right)\arrow[r] & H^1\left(\F, H^3(\cl{X},\Z_{\ell}(2))\right).
				\end{tikzcd}
			}
		\end{equation}
		Ici
\begin{align*}
H^2_0(C,\Z_{\ell}(1))&:=\on{Ker}\left(H^2(C,\Z_{\ell}(1))\longrightarrow H^2(\cl{C},\Z_{\ell}(1))\right),\\
H^4_0(X,\Z_{\ell}(2))&:=\on{Ker}\left(H^4(X,\Z_{\ell}(2))\to H^4(\cl{X},\Z_{\ell}(2))\right).
\end{align*}		
		Pour conclure, il suffit alors de d\'emontrer que la compos\'ee
\begin{equation}\label{e.suffit-surj}
\begin{tikzcd}
\on{Pic}(S)\otimes_{\Z}   H^1\left(\F, H^1(\cl{C},\Z_{\ell}(1))\right)\ar[r] & H^1\left(\F, \on{NS}(\cl{S})\otimes_{\Z}H^1(\cl{C},\Z_{\ell}(1))\right) \arrow[d, phantom, ""{coordinate, name=Z}] \arrow[d,
rounded corners,
to path={ -- ([xshift=2ex]\tikztostart.east)
|- (Z) [near start]\tikztonodes
-| ([xshift=-2ex]\tikztotarget.west)
-- (\tikztotarget)}] & \\
&H^1\left(\F, H^2(\cl{S},\Z_{\ell}(1))\otimes_{\Z_\ell} H^1(\cl{C},\Z_{\ell}(1))\right)&
\end{tikzcd}
\end{equation}
apparaissant dans le coin inf\'erieur gauche du diagramme (\ref{gros'}) est surjective. 
		
		Par \cite[Corollary V.3.28(d)]{milne1980etale}, le conoyau de la fl\`eche injective $\on{NS}(\cl{S})\otimes_{\Z}\Z_{\ell}\to H^2(\cl{S},\Z_{\ell}(1))$ est sans torsion. Comme $\rho(\cl{S})=b_2(\cl{S})$, cette injection est un isomorphisme, ce qui implique que la deuxi\`eme fl\`eche dans (\ref{e.suffit-surj}) est un isomorphisme.
		
		Comme $\on{Pic}^0_{S/\F}(\cl{\F})$ est de torsion et $\Q_{\ell}/\Z_{\ell}\otimes_{\Z}\Z_{\ell}\simeq \Q_{\ell}/\Z_{\ell}$, on a un isomorphisme $G$-\'equivariant $\on{Pic}^0_{S/\F}(\cl{\F})\otimes_{\Z}{\Z_{\ell}}\simeq \on{Pic}^0_{S/\F}(\cl{\F})\{\ell\}$ et donc le th\'eor\`eme de Lang \cite[Theorem 2]{lang1956algebraic} appliqu\'e \`a $(\on{Pic}^0_{S/\F})_{\on{red}}$ donne $H^1(\F, \on{Pic}^0_{S/\F}(\cl{\F})\otimes_{\Z}{\Z_{\ell}})=0$.
		On tensorise la suite (\ref{pic0-pic-h2}) par $\Z_{\ell}$ et on passe \`a la cohomologie galoisienne. On en d\'eduit que la fl\`eche naturelle $\on{Pic}(S)\otimes_{\Z} \Z_{\ell}\to \on{NS}(\cl{S})^G\otimes_{\Z}\Z_{\ell}$ est surjective.	Pour montrer la surjectivit\'e de l'homomorphisme (\ref{e.suffit-surj}) et conclure, il suffit alors d'\'etablir la surjectivit\'e de la fl\`eche 
		\[\on{NS}(\cl{S})^G\otimes_{\Z}   H^1\left(\F, H^1(\cl{C},\Z_{\ell}(1))\right)\longrightarrow H^1\left(\F, \on{NS}(\cl{S})\otimes_{\Z}H^1(\cl{C},\Z_{\ell}(1))\right)\]
		donn\'ee par le cup-produit en cohomologie galoisienne.
		
		Par hypoth\`ese  le groupe $G$ agit trivialement sur $\on{NS}(\cl{S})$. En \'ecrivant $\on{NS}(\cl{S})$ comme la somme directe de son sous-groupe de torsion et d'un sous-module libre, on se r\'eduit \`a v\'erifier que pour tout $m\geq 1$, la fl\`eche naturelle
		\[\phi_m:H^1\left(\F, H^1(\cl{C},\Z_{\ell}(1))\right)/\ell^m\longrightarrow H^1\left(\F, H^1(\cl{C},\mu_{\ell^m})\right)\]
		est un isomorphisme. Comme $H^2(\cl{C},\Z_{\ell}(1))\simeq \Z_{\ell}$ est sans torsion, par \cite[Lemma V.1.11]{milne1980etale} la fl\`eche naturelle \[H^1(\cl{C},\Z_{\ell}(1))/\ell^m\longrightarrow H^1(\cl{C},\mu_{\ell^m})\]
		est un isomorphisme. On a alors une  suite exacte courte
		\[
		\begin{tikzcd}
		0\ar[r] & H^1(\cl{C},\Z_{\ell}(1)) \ar[r,"\times\, \ell^m"]& H^1(\cl{C},\Z_{\ell}(1)) \ar[r]& H^1(\cl{C},\mu_{\ell^m})\ar[r]& 0.
		\end{tikzcd}
		\]
		Par passage  \`a la cohomologie galoisienne, on obtient une suite exacte
		\[
		\begin{tikzcd}
		0\ar[r]& H^1\left(\F, H^1(\cl{C},\Z_{\ell}(1))\right)/\ell^m\ar[r,"\phi_m"]& H^1\left(\F, H^1(\cl{C},\mu_{\ell^m})\right)\ar[r]& H^2\left(\F,H^1(\cl{C},\Z_{\ell}(1))\right).
		\end{tikzcd}
		\]
		La conclusion suit du fait que, d'apr\`es le lemme \ref{limit-coho}, on a
		\[H^2\left(\F,H^1(\cl{C},\Z_{\ell}(1))\right)\simeq \varprojlim_m H^2\left(\F, H^1(\cl{C},\mu_{{\ell}^m})\right)=0.\qedhere\]	
	\end{proof}

	\begin{lemma}\label{h1-cor}
		Soient $E/\F$ une extension finie et $M$ un $G$-module continu de type fini sur $\Z_{\ell}$.  La fl\`eche de corestriction $\nu:H^1(E,M) \longrightarrow H^1(\F,M)$ est surjective.
	\end{lemma}
	
	\begin{proof}
		Soit $H:=\on{Gal}(\cl{\F}/E)$, et soit $M_0$ le noyau de la fl\`eche $\Z[G/H]\otimes_{\Z} M\to M$ induite par la norme $\Z[G/H]\to \Z$. La fl\`eche $\nu$ s'identifie \`a l'homomorphisme induit $H^1\left(\F,\Z[G/H]\otimes_{\Z} M\right)\to H^1(\F,M)$. Comme $M_0\simeq \varprojlim_n M_0/\ell^n$ et $M_0/\ell^n$ est fini pour tout $n\geq 1$, le lemme \ref{limit-coho}(a) donne $H^2(\F,M_0)=0$. La suite exacte longue en cohomologie galoisienne associ\'ee \`a \[0\longrightarrow M_0 \longrightarrow \Z[G/H]\otimes_{\Z} M\longrightarrow M\longrightarrow 0\] nous permet de conclure. 
	\end{proof}

	\begin{proof}[D\'emonstration du th\'eor\`eme \ref{iota-dans-image}]
		Soient $E/\F$ une extension finie et $f:X_E\to X$ la projection. Par les lemmes \ref{compat} et \ref{hochschild-cores}(a), on a le diagramme commutatif
		\[
		\begin{tikzcd}[column sep=large]
			CH^2(X_E) \arrow[r, "\on{cl}_{X_E}"] \arrow[d,"f_*"] &  H^4(X_E,\Z_{\ell}(2)) \arrow[d,"\on{tr}_f"] & H^1\left(E,H^3(\cl{X},\Z_{\ell}(2))\right) \arrow[l, swap, "\iota_{X_E}"] \arrow[d, "H^1(\on{tr}_{\cl{f}})"]  \\
			CH^2(X) \arrow[r, "\on{cl}_X"] & H^4(X,\Z_{\ell}(2)) & H^1\left(\F,H^3(\cl{X},\Z_{\ell}(2))\right). \arrow[l, swap, "\iota_X"]   
		\end{tikzcd}
		\]
		Par les lemmes \ref{hochschild-cores}(b) et \ref{h1-cor}, la fl\`eche verticale \`a droite est surjective. Il suffit alors de montrer que l'image de $\on{cl}_{X_E}$ contient l'image de $\iota_{X_E}$. On choisit une extension $E/\F$ telle que $\on{Gal}(\cl{\F}/E)$ agit trivialement sur $\on{NS}(\cl{S})$. En rempla\c{c}ant $\F$ par $E$, on peut alors supposer que $G$ agit trivialement sur $\on{NS}(\cl{S})$. Vue la d\'ecomposition de K\"unneth (\ref{kunneth3}), la conclusion suit des lemmes \ref{lemmeA}, \ref{lemmeB} et \ref{lemmeC}.
	\end{proof}
	
	\begin{proof}[D\'emonstration du th\'eor\`eme \ref{mainthm}]
		D'apr\`es le th\'eor\`eme \ref{iota-dans-image}, l'image de l'application (\ref{tate-int3}) contient l'image de $\iota_X$. La conclusion suit de l'exactitude de la suite (\ref{hochschild}).
	\end{proof}

	\section{Preuve du th\'eor\`eme \ref{mainthm'}}
	
	\begin{lemma}\label{h1-pic0}
		Soit $V$ une vari\'et\'e projective, lisse et g\'eom\'etriquement connexe sur $\F$. Alors on a un isomorphisme $G$-\'equivariant $H^1(\cl{V},\Z_{\ell}(1))\simeq T_{\ell}((\on{Pic}^0_{\cl{V}/\cl{\F}})_{\on{red}})$.
	\end{lemma}
	
	\begin{proof}
		Pour tout $n\geq 1$, la suite de Kummer donne $H^1(\cl{V},\mu_{\ell^n})\simeq \on{Pic}(\cl{V})[\ell^n]$ qui est un isomorphisme de $G$-modules. Par passage \`a la limite projective sur $n\geq 1$, on obtient $H^1(\cl{V},\Z_{\ell}(1))\simeq T_{\ell}(\on{Pic}(\cl{V}))$.  On a une suite exacte courte
		\[0\longrightarrow \on{Pic}_{\cl{V}}^0(\cl{\F})\set{\ell}\longrightarrow \on{Pic}(\cl{V})\set{\ell}\longrightarrow M\longrightarrow 0,\] o\`u $M$ est un $\Z$-module de type fini. Comme $M$ est fini, $T_{\ell}(M)=0$. Pour tout $n\geq 1$, les foncteurs de $\ell^n$-torsion et de limite projective sont exacts \`a gauche, donc le foncteur $T_{\ell}(-)$ est exact \`a gauche. On conclut que l'on a \[H^1\left(\cl{V},\Z_{\ell}(1)\right)\simeq T_{\ell}\left(\on{Pic}(\cl{V})\right)\simeq T_{\ell}\left(\left(\on{Pic}^0_{\cl{V}/\cl{\F}}\right)_{\on{red}}\right).\qedhere\]
	\end{proof}
	
	\begin{lemma}\label{hyp-pic}
		Soient $C$ et $S$ deux vari\'et\'es g\'eom\'etriquement connexes, projectives et lisses sur $\F$, de dimension $1$ et $2$ respectivement et $X:=C\times_{\F} S$. On suppose que l'on a
		\[
		{\rm Hom}_{G}\left(\on{NS}(\cl{S})\{\ell\}, J_C(\cl{\F})\set{\ell}\right)=0\quad\text{et}\quad \on{Hom}_{\F-{\rm gr}}\left(\left(\on{Pic}^0_{S/\F}\right)_{\on{red}},J_C\right)=0.
		\]
		Alors
		\[\left(H^3\left(\cl{S},\Z_{\ell}(2)\right)\otimes_{\Z_{\ell}} H^1\left(\cl{C},\Z_{\ell}\right)\right)^G=0.\]
	\end{lemma}

	\begin{proof}
		D'apr\`es \cite[(8.12)]{grothendieck1968brauer3} ou \cite[Prop. 5.2.10]{ctskobrauer},
		on a un isomorphisme de $G$-modules
		\[H^3(\cl{S},\Z_{\ell}(2))_{\on{tors}}  \simeq \on{Hom}\left(\on{NS}(\cl{S})\set{\ell},\Q_{\ell}/\Z_{\ell}(1)\right).\]
		D'apr\`es le lemme \ref{h1-pic0}, on a	$H^1(\cl{C},\Z_{\ell}(1))\simeq T_{\ell}(J_C(\cl{\F}))$,
		donc
		\[H^3(\cl{S},\Z_{\ell}(2))_{\on{tors}}\otimes_{\Z_{\ell}} H^1(\cl{C},\Z_{\ell})\simeq \on{Hom}\left(\on{NS}(\cl{S})\set{\ell},J_C(\cl{\F})\set{\ell}\right).\] 
		On en d\'eduit
		\begin{equation}\label{jac-pic1}
			\left(H^3(\cl{S},\Z_{\ell}(2))_{\on{tors}}\otimes_{\Z_{\ell}} H^1(\cl{C},\Z_{\ell})\right)^G=0.
		\end{equation}
		
		Soit $A:=(\on{Pic}^0_{S/\F})_{\on{red}}$. Le th\'eor\`eme de Tate \cite[Main Theorem]{tate1966endomorphisms} donne
		\[\on{Hom}_{\F-\on{gr}}(A,J_C)\otimes_{\Z}\Z_{\ell}\simeq \on{Hom}_G\left(T_{\ell}(\cl{A}),T_{\ell}(J_{\cl{C}})\right).\]
		
		On a alors des isomorphismes
		\begin{align*}
			\on{Hom}_{\F-\on{gr}}(A,J_C)\otimes_{\Z}\Q_{\ell} &\simeq \on{Hom}_G\left(V_{\ell}(\cl{A}),V_{\ell}(J_{\cl{C}})\right) \\
			&\simeq \on{Hom}_G\left(H^1(\cl{S},\Q_{\ell}(1)),H^1(\cl{C},\Q_{\ell}(1))\right) \\
			&\simeq \left(H^1(\cl{S},\Q_{\ell}(1))^{\vee}\otimes_{\Q_{\ell}} H^1(\cl{C},\Q_{\ell}(1))\right)^G \\
			&\simeq \left(H^3(\cl{S},\Q_{\ell}(1))\otimes_{\Q_{\ell}} H^1(\cl{C},\Q_{\ell}(1))\right)^G\\
			&\simeq \left(H^3(\cl{S},\Q_{\ell}(2))\otimes_{\Q_{\ell}} H^1(\cl{C},\Q_{\ell})\right)^G
		\end{align*}
		Ici le premier isomorphisme suit du fait que les $\Z_{\ell}$-modules $T_{\ell}(\cl{A})$ et $T_{\ell}(J_{\cl{C}})$ sont de type fini, le deuxi\`eme isomorphisme vient du lemme \ref{h1-pic0}, et l'isomorphisme $G$-\'equivariant $H^1(\cl{S},\Q_{\ell}(1))^{\vee}\simeq H^3(\cl{S},\Q_{\ell}(1))$ est donn\'e par la dualit\'e de Poincar\'e $\ell$-adique pour la surface $S$; voir \cite[Th\'eor\`eme (2.3)]{weil1}. 
		Par hypoth\`ese, $\on{Hom}_{\F-\on{gr}}(A,J_C)=0$, donc 
		\[\left(H^3(\cl{S},\Q_{\ell}(2))\otimes_{\Q_{\ell}} H^1(\cl{C},\Q_{\ell})\right)^G=0.\] 
		Il s'en suit que
		\begin{equation}\label{jac-pic2}
			\left((H^3(\cl{S},\Z_{\ell}(2))/\on{tors})\otimes_{\Z_{\ell}} H^1(\cl{C},\Z_{\ell})\right)^G=0.	
		\end{equation} 
		On tensorise la suite exacte courte de $G$-modules
		\[0\longrightarrow H^3(\cl{S},\Z_{\ell}(2))_{\on{tors}}\longrightarrow H^3(\cl{S},\Z_{\ell}(2))\longrightarrow H^3(\cl{S},\Z_{\ell}(2))/\on{tors}\longrightarrow 0\]
		par le $\Z_{\ell}$-module libre $H^1(\cl{C},\Z_{\ell})$ et on passe aux $G$-invariants. La conclusion suit de la combinaison de (\ref{jac-pic1}) et (\ref{jac-pic2}).
	\end{proof}
	
	\begin{lemma}\label{tate-ginv}
		Soient $C$ et $S$ deux vari\'et\'es g\'eom\'etriquement connexes, projectives et lisses sur $\F$, de dimension $1$ et $2$ respectivement et $X:=C\times_{\F} S$. On suppose que l'on a $b_2(\cl{S})=\rho(\cl{S})$ et
		\[
		{\rm Hom}_{G}\left(\on{NS}(\cl{S})\{\ell\}, J_C(\cl{\F})\set{\ell}\right)=0\quad\text{et}\quad \on{Hom}_{\F-{\rm gr}}\left((\on{Pic}^0_{S/\F})_{\on{red}},J_C\right)=0.
		\]
		Alors l'application cycle
		\begin{equation*}\label{eq.tate-galois}\on{cl}_X:CH^2(X)\otimes\Z_{\ell}\longrightarrow H^4\left(\cl{X},\Z_{\ell}(2)\right)^G
		\end{equation*}
		est surjective.

	\end{lemma}
	
	\begin{proof}
		La formule de K\"unneth en cohomologie $\ell$-adique nous donne un isomorphisme $G$-\'equivariant
		\[H^4(\cl{X},\Z_{\ell}(2))\simeq H^4(\cl{S},\Z_{\ell}(2))\oplus \left[H^3(\cl{S},\Z_{\ell}(2))\otimes_{\Z_{\ell}} H^1(\cl{C},\Z_{\ell})\right]\oplus H^2(\cl{S},\Z_{\ell}(1)).\]
		Par le lemme (\ref{hyp-pic}) on obtient
		\[H^4(\cl{X},\Z_{\ell}(2))^G\simeq H^4(\cl{S},\Z_{\ell}(2))^G\oplus  H^2(\cl{S},\Z_{\ell}(1))^G.\]
		On a le carr\'e commutatif	
		\[
		\begin{tikzcd}
			(CH^2(S)\otimes_{\Z}\Z_{\ell})\oplus (\on{Pic}(S)\otimes_{\Z}\Z_{\ell})\arrow[r] \arrow[d] & CH^2(X)\otimes_{\Z}\Z_{\ell} \arrow[d]\\
			H^4(\cl{S},\Z_{\ell}(2))^G\oplus  H^2(\cl{S},\Z_{\ell}(1))^G \arrow[r,"\sim"] & H^4(\cl{X},\Z_{\ell}(2))^G,
		\end{tikzcd}
		\]	
		o\`u les fl\`eches verticales sont les applications cycle et les fl\`eches horizontales sont induites par l'image r\'eciproque et le cup-produit. Comme $\F$ est fini, d'apr\`es les estimations de Lang--Weil, $S$ admet un $0$-cycle de degr\'e $1$ (voir \cite[1.5.3. Lemme 1]{soule1984groupes}) et donc la fl\`eche de degr\'e
		$CH^2(S)\otimes_{\Z}\Z_{\ell}\to H^4(\cl{S},\Z_{\ell}(2))^G=\Z_{\ell}$
		est surjective.	Comme $b_2(\cl{S})=\rho(\cl{S})$, par \cite[Corollary V.3.28(d)]{milne1980etale} la fl\`eche $\on{Pic}(S)\otimes_{\Z}\Z_{\ell}\to H^2(\cl{S},\Z_{\ell}(1))^G$ est aussi surjective. On conclut alors que la fl\`eche verticale de droite est surjective, comme voulu.
	\end{proof}

	\begin{proof}[D\'emonstration du th\'eor\`eme \ref{mainthm'}]
		(a)  L'application (\ref{tate-int1}) est surjective d'apr\`es le lemme \ref{tate-ginv}, donc le th\'eor\`eme \ref{mainthm} implique que l'application  (\ref{tate-int3}) est surjective.
		
		(b) La partie (a) et un th\'eor\`eme de Kahn \cite[Thm. 5.10]{ctscavia1} montrent que le groupe $H^3_{\on{nr}}(\F(X),\Q_{\ell}/\Z_{\ell}(2))$ est divisible. Par hypoth\`ese, il existe une vari\'et\'e projective et lisse $V$ de dimension $\leq 1$ et un morphisme $f:V\to S$ tels que pour tout corps alg\'ebriquement clos $\Omega$ contenant $\F$ l'homomorphisme d'image directe $f_*:CH_0(V_{\Omega})_{\Q}\to CH_0(S_{\Omega})_{\Q}$ est surjectif. Donc, pour tout $m\in \Z$, tous $P_1,P_2\in S(\Omega)$ tels que $mP_1$ est rationnellement \'equivalent \`a $mP_2$, et tout $P\in C(\Omega)$, les $0$-cycles $m\cdot (P,P_1)$, $m\cdot(P,P_2)\in X(\Omega)$ sont	rationnellement \'equivalents. On en d\'eduit que l'homomorphisme
\[(\on{id}_C\times f)_*:CH_0\left(C_{\Omega}\times V_{\Omega}\right)_{\Q}\longrightarrow CH_0\left(X_{\Omega}\right)_{\Q}\]
est surjectif, et donc que $CH_0(X)_{\Q}$ est support\'e en dimension $2$. Un argument de correspondances bien connu  \cite[Proposition~3.2]{colliot2013cycles} implique alors que le groupe $H^3_{\on{nr}}(\F(X),\Q_{\ell}/\Z_{\ell}(2))$ est annul\'e par
		un entier positif. Comme le groupe  $H^3_{\on{nr}}(\F(X),\Q_{\ell}/\Z_{\ell}(2))$ est divisible, il est nul. 
	\end{proof}

%%%%%%%%%%%%%%%%%%%%%
% References
%%%%%%%%%%%%%%%%%%%%%

\BibliographyTOCName{R\'ef\'erences}

\end{document}